\documentclass{amsart}

\usepackage[utf8]{inputenc}
\usepackage{amssymb}
\usepackage{amsfonts}
\usepackage{amsmath}
\usepackage{amsthm}
\usepackage{enumerate}
\usepackage{mathrsfs}
\usepackage[pdftex]{graphicx}
\usepackage{wrapfig}
\usepackage{comment}
\usepackage{graphicx}
\usepackage{mathtools}
\usepackage{todonotes}
\usepackage{tikz-cd}
\usepackage{mathrsfs}
\usepackage{hyperref}
\usepackage{cleveref}
\usepackage{xcolor}
\hypersetup{
	colorlinks,
	linkcolor={red!50!black},
	citecolor={blue!50!black},
	urlcolor={blue!50!black}
}

\makeatletter
\newtheorem*{rep@thm}{\rep@title}
\newcommand{\newreptheorem}[2]{%
	\newenvironment{rep#1}[1]{%
		\def\rep@title{#2 \ref{##1}}%
		\begin{rep@thm}}%
		{\end{rep@thm}}}
\makeatother

\newtheorem{thm}{Theorem}[section]
\newtheorem{cor}[thm]{Corollary}
\newtheorem{conj}[thm]{Conjecture}
\newtheorem{prop}[thm]{Proposition}

\newtheorem{lem}[thm]{Lemma}

\newreptheorem{thm}{Theorem}
\newreptheorem{lem}{Lemma}

\theoremstyle{definition}
\newtheorem{defn}[thm]{Definition}

\theoremstyle{remark}

\newcommand{\N}{\mathbb N}
\newcommand{\Z}{\mathbb Z}
\newcommand{\Q}{\mathbb Q}
\newcommand{\R}{\mathbb R}
\newcommand{\C}{\mathbb C}
\newcommand{\comp}[1]{{#1}^{\mathsf{c}}}
\let\conj\overline
\let\clos\overline

\let\emptyset\varnothing

\let\epsilon\varepsilon
\let\phib\phi
\let\phi\varphi

\newcommand*{\mb}[1]{\ensuremath{\mathbb{#1}}}
\newcommand*{\mc}[1]{\ensuremath{\mathcal{#1}}}
\newcommand*{\mf}[1]{\ensuremath{\mathfrak{#1}}}
\newcommand*{\ms}[1]{\ensuremath{\mathscr{#1}}}
\newcommand*{\bs}{\backslash}
\newcommand*{\dprod}{\displaystyle\prod}

\DeclareMathOperator{\GL}{GL} 

\DeclareMathOperator{\Hom}{Hom}
\DeclareMathOperator{\Spec}{Spec}
\DeclareMathOperator{\Lie}{Lie}
\DeclareMathOperator{\End}{End}
\DeclareMathOperator{\Tr}{Tr}
\DeclareMathOperator{\Res}{Res}
\DeclareMathOperator{\Fil}{Fil}
\DeclareMathOperator{\dR}{dR}
\DeclareMathOperator{\Gr}{Gr}
\DeclareMathOperator{\Aut}{Aut}
\DeclareMathOperator{\der}{der}
\DeclareMathOperator{\ad}{ad}
\DeclareMathOperator{\Gal}{Gal}

\DeclarePairedDelimiter{\norm}{\lVert}{\rVert}
\DeclarePairedDelimiter{\abs}{\lvert}{\rvert}
\DeclarePairedDelimiter{\inner}{\langle}{\rangle}
\DeclarePairedDelimiter{\set}{\{}{\}}
\DeclarePairedDelimiter{\paren}{(}{)}

\title{Heights of Special Points on Quaternionic Shimura Varieties}
\author[Roy Zhao]{Roy Zhao}
\address{Department of Mathematics, University of California, Berkeley, CA, USA.}
\email{rhzhao@berkeley.edu}

\begin{document}
	
	\begin{abstract}
	    Let $B/F$ be a quaternion algebra over a totally real number field. We give an explicit formula for heights of special points on the quaternionic Shimura variety associated with $B$ in terms of Faltings heights of CM abelian varieties. Special points correspond to CM fields $E$ and partial CM-types $\phib \subset \Hom(E, \C)$. We then show that our height is compatible with the canonical height of a partial CM-type defined by Pila, Shankar, and Tsimerman in \cite{Pil21}. This gives another proof that the height of a partial CM-type is bounded subpolynomially in terms of the discriminant of $E$.
	\end{abstract}
	
	\maketitle
	
	\section{Introduction}
	
	In this article we study heights of special points on quaternionic Shimura varieties, motivated by the recent proof of the Andr\'e--Oort conjecture for Shimura varieties by \cite{Pil21}. Let $B/F$ be a quaternion algebra over a totally real field. Special points on an associated quaternionic Shimura variety $X$ correspond to totally imaginary quadratic extensions $E/F$ lying inside $B$. We give a direct formula for the height of a special point in terms of Faltings heights of CM-types of $E$ and explicit logarithms of discriminants. The splitting behavior of $B$ at infinity gives rise to a partial CM-type $\phib \subset \Hom(E, \C)$ such that $\phib \cap \conj{\phib} = \emptyset$. We show that the height of a special point on this quaternionic Shimura variety is compatible with the height of the partial CM-type $h(\phib)$ defined by \cite{Pil21}. Our direct formula gives another proof that the heights of special points of Shimura varieties are bounded in terms of discriminants of number fields. Our formula for $h(\phib)$ also differs from the one given by \cite{Pil21}. The formula given by \cite{Pil21} expresses the height $h(\phib)$ in terms of Faltings heights of CM-types of CM-fields $E'$, where the relative discriminant of $E'/E$ is controlled. Our formula expresses $h(\phib)$ in terms of CM-types of $E$ only. The case when $B$ is split at only one archimedean place of $F$ was proved by \cite{Yua18}, and we recover their formula in that setting. To prove our formula, we follow the strategy of \cite{Yua18} by comparing a Kodaira--Spencer map on $X$ with that of a related PEL-type Shimura variety $X'$.
	
	\subsection{Statement of Main Theorem}
	
	Let $E$ be a CM field, and $F$ be its totally real subfield, so that $[E:F] = 2$. Set $g \coloneqq [F:\Q]$. Let $\phib \subset \Hom(E, \C)$ be a partial CM-type, meaning that $\phib \cap \conj{\phib} = \emptyset$. Write $\Sigma \subset \Hom(F, \R)$ for the restriction of $\phib$ to $F$. Suppose that $B/F$ is a quaternion algebra with the following properties:
 \begin{enumerate}
     \item There exists an embedding $E \hookrightarrow B$;
     \item The ramification set of $B$ at infinity is $\Sigma^c$;
     \item If $B$ is ramified at a finite prime $\mf{p}$ of $F$, then $E$ is also ramified over $\mf{p}$.
 \end{enumerate}
 We define the algebraic group $G$ over $\Q$ as
	\[G \coloneqq \Res_{F/\Q} B^\times.\]
	For each compact open subgroup $U \subset G(\mb{A}_f)$, we obtain a (quaternionic) Shimura variety $X_U$, defined over a number field $E_X$, with complex uniformization given by:
	\[X_U(\C) = G(\Q) \backslash (\mc{H}^\pm)^\Sigma \times G(\mb{A}_f)/U,\]
	where $\mc{H}^\pm$ is the upper and lower complex half-planes. This is a Shimura variety of abelian type, and by utilizing ideas from \cite{Car86} and \cite{Pap13}, we can construct a regular integral model $\mc{X}_U$ for $X_U$ over $\Spec \mc{O}_{E_X}$.
	
	Let $\widehat{\mc{L}_U}$ be the arithmetic Hodge bundle of $\mc{X}_U$, which consists of a line bundle $\mc{L}_U$ on $\mc{X}_U$ and a Hermitian metric given by:
	\[\norm*{\bigwedge_{\sigma \in \Sigma} dz_\sigma} \coloneqq \prod_{\sigma \in \Sigma} 2\mathrm{Im}(z_\sigma),\]
        where the $z_\sigma$ are given by the complex uniformization of $X_U$. When $U$ is sufficiently small, the Hodge bundle is simply the canonical bundle $\mc{L}_U = \omega_{\mc{X}_U/\mc{O}_{E_X}}$. The precise definition of $\widehat{\mc{L}_U}$ is given in Section \ref{sec:Hodge bundle}.
	
	Let $P_U \subset X_U(\clos{\Q})$ be a special point arising from the embedding $E \hookrightarrow B$, and let $\clos{P_U}$ be the closure of this point in $\mc{X}_U$, which we will also denote by $P_U$ by abuse of notation. The height of this point relative to $\widehat{\mc{L}_U}$ is the Arakelov height:
	\[h_{\widehat{\mc{L}_U}}(P_U) \coloneqq \frac{1}{[\Q(P_U):\Q]} \widehat{\deg}(\widehat{\mc{L}_U}|_{P_U}).\]
	
	Finally, let $\Phi$ be a full CM-type, and let $h(\Phi)$ be the Faltings height of an abelian variety with complex multiplication by $(\mc{O}_E, \Phi)$. Let $d_\phib, d_{\conj{\phib}}$ and $d_\Sigma \coloneqq d_{\phib \sqcup \conj{\phib}}$ be certain absolute discriminants of $\phib, \conj{\phib}$, and $\phib \sqcup \conj{\phib}$. These are defined in detail in Section \ref{sec:Faltings height}.
 
    Let $\mf{d}_{E/F}$ denote the relative discriminant of the extension $E/F$. There is a reflex norm $N_{F/E_X} \colon F \to E_X$ defined by $N_{F/E_X}(x) = \prod_{\sigma \in \Sigma} \sigma(x)$. Let $d_{E/F, \Sigma} \in \Z$ be the positive generator of $N_{E_X/\Q}(N_{F/E_X}(\mf{d}_{E/F}))$. Let $d_{E/F} \in \Z$ be the positive generator of $N_{F/\Q}(\mf{d}_{E/F})$ and let $d_F$ be the absolute discriminant of $F$. Let $d_B$ be the positive generator of norm from $F$ to $\Q$ of the product of all the finite places of $\mc{O}_F$ over which $B$ ramifies.
    
    With these definitions in place, we can now state our main theorem.
	
	\begin{thm}\label{thm:main theorem 1}
	    Suppose that $U = \prod_v U_v$ is a maximal compact subgroup of $G(\mb{A}_f)$. Then
        \begin{align*}
            \frac{1}{2}h_{\widehat{\mc{L}_U}}(P_U) =& \frac{1}{2^{\abs{\comp{\Sigma}}}} \sum_{\Phi \supset \phib} h(\Phi) - \frac{\abs{\comp{\Sigma}}}{g2^g} \sum_\Phi h(\Phi)\\
            & + \frac{1}{8} \log d_{E/F, \Sigma}d_\Sigma^{-1} + \frac{1}{4}\log d_\phib d_{\conj{\phib}} + \frac{1}{4g} \log d_Bd_\Sigma + \frac{\abs{\Sigma}}{4g} \log d_F.
            \end{align*}
		The first summation is over all full CM-types which contain $\phib$, and the second summation is over all full CM-types of $E$.
	\end{thm}
	
	Additionally, if $\abs{\phib} = 1$, then $E_X = F$, and we have that $d_\Sigma = d_{E/F} = d_{E/F, \Sigma}$ and $d_\phib = d_{\conj{\phib}} = 1$. As a result, the expression for $\frac{1}{2}h_{\widehat{\mc{L}_U}}(P_U)$ simplifies, and we recover \cite[Thm. 1.6]{Yua18}, where the factor of $g$ is due to different normalizing factors of $h_{\widehat{\mc{L}_U}}(P_U)$.

    In the Pila--Zannier method, an essential step involves bounding the height of a CM point in terms of the discriminant of its splitting field. For general Shimura varieties, a CM point $P$ is associated with a partial CM-type $\phib$ of a CM-field $E$. In \cite{Pil21}, a canonical height $h(\phib)$ is introduced for such $\phib$, and this height $h(\phib)$ is shown to be equal to the height $h(P)$ of the associated CM point $P$, up to $O(\log d_E)$. Then, the height $h(\phib)$ is bounded in terms of the discriminant $d_E$, showing that $h(P)$ is as well. The precise definition of $h(\phib)$ can be found in Section \ref{sec:Andre Oort}. Our second main result establishes the compatibility between $h_{\widehat{\mc{L}_U}}(P_U)$ and $h(\phib)$.
    
    \begin{thm}\label{thm:main theorem 2}
        \[h(\phib) = \frac{1}{2}h_{\widehat{\mc{L}_U}}(P_U) + O(\log d_E).\]
    \end{thm}

    With the combination of Theorem \ref{thm:main theorem 1} and Theorem \ref{thm:main theorem 2}, we obtain the following corollary. This is an ingredient used in the Pila--Zannier method, and serves as one of the results of \cite{Pil21}.

    \begin{cor}\label{cor:bound partial CM height}
        For all $\epsilon > 0$, there exists a positive constant $c$ depending only on $[E:\Q]$ such that
        \[h(\phib) \le c \cdot d_E^\epsilon\]
        for all partial CM-types of $E$.
    \end{cor}
    \begin{proof}
        By \cite[Cor. 3.3]{Tsi18}, the Faltings heights $h(\Phi)$ of full CM-types are bounded subpolynomially by $d_E$. Each of the discriminants $d_{E/F, \Sigma}, d_\Sigma, d_\phib, d_F$ are smaller than $d_E$, and $d_B \le d_E$ since we specified that the ramification set of $E$ contains the ramification set of $B$. Thus, each of the logarithm terms are bounded by $\log d_E$, which is also subpolynomial in $d_E$.
    \end{proof}

    This result differs from that of \cite{Pil21} because in our case, we are able to express $h(\phib)$ in terms of CM-types of $E$, whereas in \cite{Pil21}, they express $h(\phib)$ in terms of CM-types $\Phi^\sharp$ of CM-fields $E^\sharp$ containing $E$, whose relative discriminant over $E$ can be bounded.

        \subsection{Motivation}
	
	The Andr\'e--Oort conjecture states that if $S$ is a Shimura variety and $V \subset S$ is a subvariety with a dense subset of special points, then $V$ itself must be a special subvariety. Definitions and properties for Shimura varieties can be found in \cite[Sec. 4]{Mil05} and for special subvarieties in \cite[Sec. 2]{Daw18}. The Andr\'e--Oort conjecture was originally proven by Andr\'e in the case when $S$ is the product of two modular curves \cite{And98} and later for arbitrary products of modular curves by Pila (see \cite{Pil11}). Pila and Tsimerman (see \cite{Pil14}) were able to extend this strategy to prove the conjecture unconditionally for $\mc{A}_g$, the coarse moduli space of principally polarized abelian varieties of dimension $g$, when $g \le 6$. The averaged Colmez Conjecture, proven by \cite{And18} and \cite{Yua18}, allowed Tsimerman to give an unconditional proof (see \cite{Tsi18}) of the conjecture for $\mc{A}_g$ for all $g$. Moreover, following results of \cite{Bin23}, Pila, Shankar, and Tsimerman recently announced an unconditional proof of the conjecture for all Shimura varieties in \cite{Pil21}. There is also a version of the Andr\'e--Oort conjecture for mixed Shimura varieties that was reduced to the pure Shimura varieties case by Gao (see \cite{Gao16}).

	Many of the recent results on the Andr\'e--Oort conjecture were proven using a strategy that was initially proposed by Zannier, by using the theory of o-minimality and a point counting theorem of Pila and Wilkie from \cite{Pil06}, combined with estimates on the sizes of certain Galois orbits. This strategy was shown to be viable when Pila and Zannier used it to reprove the Manin--Mumford conjecture in \cite{Zan08}. And it is using this Pila--Zannier strategy that the conjecture was proven for products of the modular curve (see \cite{Pil11}), the moduli space $\mc{A}_g$ (see \cite{Pil14, Tsi18}), and for all Shimura varieties (see \cite{Pil21}). The strategy can be split up into three main ingredients.
	\begin{enumerate}
	    \item The first ingredient is the Pila--Wilkie point counting theorem from \cite{Pil06}. It gives an upper bound subpolynomial in height on the rational points of the transcendental component of definable sets. The fundamental domains of the universal covering map of a Shimura variety are definable in an o-minimal structure. This was shown for $\mc{A}_g$ by Peterzil and Starchenko (see \cite{Pet13}), and for arbitrary Shimura varieties by Klinger, Ullmo, and Yafaev (see \cite{Kli13}). A sharpened version proven by Binyamini (see \cite{Bin22}) was necessary for general Shimura varieties.
	    
	    \item The second ingredient is a lower bound polynomial in height for the size of Galois orbits of special points. This is done in two steps. The first step to provide a lower bound polynomial in height for the discriminant of certain endomorphism algebras. This was done by Pila and Tsimerman (see \cite{Pil13}) for $\mc{A}_g$, and by Binyamini, Schmidt, and Yafaev (see \cite{Bin23}) for arbitrary Shimura varieties. The second step is to provide a lower bound subpolynomial in terms of discriminant of those same algebras. Tsimerman proves this for $\mc{A}_g$ (see \cite{Tsi18}) by combining a result of Masser and W\"ustholz (see \cite{Mas95}), and the averaged Colmez conjecture, which was proven independently by Andreatta, Goren, Howard, and Madapusi-Pera (see \cite{And18}), and Yuan and Zhang (see \cite{Yua18}). For general Shimura varieties, Binyamini, Schmidt, and Yafaev (see \cite{Bin23}) reduce this to proving the existence of a canonical height on Shimura varieties and that the height of special points is bounded subpolynomially in terms of the discriminant of certain number fields. A proof of this result was recently announced by Pila, Shankar, and Tsimerman (see \cite{Pil21}).
	    
	    \item The third ingredient is the Ax--Lindemann theorem. The first two ingredients combine to show that Galois orbits of special points must lie in the algebraic part of the fundamental domain. Then, the Ax--Lindemann theorem tells us that these algebraic parts are precisely special subvarieties. This was proven for $\mc{A}_g$ by Pila and Tsimerman (see \cite{Pil14}), and by Mok, Pila, and Tsimerman for all Shimura varieties (see \cite{Mok19}).
	\end{enumerate}
	
	For this article, we are interested in the contribution of \cite{Pil21} in showing that their canonical height for CM points on Shimura varieties is bounded in terms of discriminants of their splitting fields. The first step is to systematically define a Weil height function for any arbitrary Shimura variety. Given a Shimura variety $\mathrm{Sh}_K(G, X)$, take a $\Q$-representation $G \to \GL(V)$ and lattice $\Lambda \subset V$. By the Riemann--Hilbert correspondence over $p$-adic local fields, given by \cite{Dia22}, we get a filtered automorphic vector bundle with connection $(_{\dR}V, \Fil^\bullet, \nabla)$, which is defined over the reflex field of $\mathrm{Sh}_K(G, X)$ and all of its $p$-adic places. Then the plan is to define an adelic norm on $\Gr^\bullet_{\dR} V$, which would give rise to an Arakelov height function on $\mathrm{Sh}_K(G, X)$. At the archimedean places, this representation admits a polarization $\psi\colon V \times V \to \Q$ and the norm can be defined as the Hodge norm $\psi(v, h(i)v)$. Over the finite places, the crystalline norm is used when the representation is crystalline and an alternative intrinsic norm is used at the finitely many other places. This height is compatible in the sense that if $(G_1, X_1) \to (G_2, X_2)$ is a map of Shimura datum with $\rho_i$ a representation of $G_i$ compatible with this morphism, then the height of a point of $\mathrm{Sh}_{K_2}(G_2, X_2)$ with respect to $\rho_2$ is equal to the height of a point of $\mathrm{Sh}_{K_1}(G_1, X_1)$ with respect to $\rho_1$. This height also coincides with the Faltings height for $\mc{A}_g$.
	
	With this height defined, the next step is to bound the height of special points in terms of discriminants of certain number fields. If $(T, x) \subset (G, X)$ is a Shimura sub-datum of a special point, then $T$ splits over a CM field $E/F$. Given a representation $\rho \colon G \to \GL(V)$ and restricted representation $\rho_x \colon T \to G \to \GL(V)$, we can map $(T, x, \rho_x) \to (\Res_{F/\Q} E^\times/F^\times, x, \rho_\phib)$ to another Shimura datum where $\phib \subset \Hom(E, \C)$ is a partial CM-type and $\rho_\phib$ is a representation of $\Res_{F/\Q} E^\times/F^\times$ in terms of $\phib$. The height $h(\phib)$ is defined as the height of $(\Res_{F/\Q} E^\times/F^\times, x, \rho_\phib)$. By the compatibility of heights, the problem is reduced to bounding the height $h(\phib)$. Given a set of disjoint CM-fields and partial CM-types $\set{(E_i, \phib_i)}_{i = 1}^t$ they are able to express a linear combination of $h(\phib_i)$ in terms of a linear combination of heights of full CM-types $\Phi^S$ of $E^S = \prod_{i \in S} E_i$ for various subsets $S \subset \{1, 2, \dots, t\}$. By ranging over all subsets $S$, we can express the height of an individual $h(\phib_i)$ as a 
 linear combination of Faltings heights of CM-types $\Phi^S$ of $E^S$, where $S$ varies over all possible subsets of $\{1, 2, \dots, t\}$. The height of the full CM-type $h(\Phi^S)$ is bounded subpolynomially in terms of the discriminant $d_{E^S}$.  Choosing $E_i$ carefully, we can bound the relative discriminant $\mf{d}_{E^S/E_i}$, completing the proof.
 
    Instead of expressing the height $h(\phib)$ in terms of heights of CM-types over the many $E^S$, we give an expression in terms of CM-types of $E$ only.
 
	\subsection{Idea of the Proof}
	
	The idea is similar to that of \cite{Yua18}. We use their decomposition of Faltings heights of a CM abelian variety $h(\Phi)$ into constituent parts $h(\Phi, \tau)$, one for each archimedean place $\tau \in \Phi$. The constituent parts are related to the full CM-type by the formula
	\[h(\Phi) - \sum_{\tau \in \Phi} h(\Phi, \tau) = \frac{-1}{4[E_\Phi:\Q]} \log (d_\Phi d_{\conj{\Phi}}),\]
	where $d_\Phi, d_{\conj{\Phi}}$ are discriminants associated with $\Phi, \conj{\Phi}$ respectively, and $E_\Phi$ is the reflex field of $\Phi$. Moreover, if $(\Phi_1, \Phi_2)$ are nearby CM-types of $E$ in that they differ only at a single place $\tau_i = \Phi_i \bs (\Phi_1 \cap \Phi_2)$, then \cite{Yua18} proves that the quantity
	\[h(\Phi_1, \tau_1) + h(\Phi_2, \tau_2)\]
	is the same across any choice of nearby CM-types.
	
	We define the group
	\[G'' \coloneqq \Res_{F/\Q} (B^\times \times E^\times)/F^\times,\]
	where $F$ embeds diagonally as $a \mapsto (a, a^{-1})$. We can construct a norm $N\colon G'' \to \Res_{F/\Q} \mb{G}_m$ and define the group
	\[G' \coloneqq G'' \times_{\mb{G}_m} \Res_{F/\Q} \mb{G}_m\]
	consisting of elements $G''$ with norm lying in $\Q^\times$. If $\phib$ is a partial CM-type and $\phib'$ is a complementary partial CM-type in that $\phib \sqcup \phib'$ constitute a full CM-type, then we can construct morphisms $h' \colon \C^\times \to G'(\R)$ and $h''\colon \C^\times \to G''(\R)$. They give rise to Shimura datum and Shimura varieties $X'_{U'}$ and $X''_{U''}$ for compact open subgroups $U' \subset G'(\mb{A}_f)$ and $U'' \subset G''(\mb{A}_f)$ with complex uniformizations
	\[X'_{U'}(\C) = G'(\Q) \backslash (\mc{H}^\pm)^\Sigma \times G'(\mb{A}_f)/U'\]
	and
	\[X''_{U''}(\C) = G''(\Q) \backslash (\mc{H}^\pm)^\Sigma \times G''(\mb{A}_f)/U''\]
	They have canonical models defined over the same reflex field $E_{X'} = E_{X''}$.
	
	The Shimura variety $X'_{U'}$ is of PEL type and has an integral model $\mc{X}'_{U'}$ by \cite{Car86} and \cite{Pap13}. The pair $(\phib, \phib')$ gives rise to a point $P'_{U'} \in X'_{U'}$ which parametrizes an abelian variety isogenous to a product $A_1 \times A_2$ of abelian varieties, one with complex multiplication of type $\phib \sqcup \phib'$ and the other with complex multiplication of type $\conj{\phib} \sqcup \phib'$. After defining a suitable metric on $\omega_{\mc{X}'_{U'}/\mc{O}_{E_{X'}}}$, the Kodaira--Spencer isomorphism on $X'_{U'}$ gives us an equality of heights
	\[h_{\widehat{\omega_{\mc{X}'_{U'}/\mc{O}_{E_{X'}}}}}(P'_{U'}) = \sum_{\tau \in \phib} \paren*{h(\phib \sqcup \phib', \tau) + h(\conj{\phib} \sqcup \phib', \conj{\tau})}.\]
	
	Now the idea is to relate $\omega_{\mc{X}_U/\mc{O}_{E_X}}$ and $\omega_{\mc{X}'_{U'}/\mc{O}_{E_{X'}}}$. We do this by mapping both $X_U$ and $X'_{U'}$ into the third Shimura variety $X''_{U''}$ so that the points $P_U$ and $P'_{U'}$ have the same image $P''_{U''} \in X''_{U''}(\clos{\Q})$. We represent both canonical bundles in terms of deformations of $p$-divisible groups $\mc{H}_U$ and $\mc{H}'_{U'}$ over $\mc{X}_U$ and $\mc{X}'_{U'}$ respectively, and then relate those $p$-divisible groups to a $p$-divisible group $\mc{H}''_{U''}$ over $\mc{X''}_{U''}$. After showing all this, we get that
	\[h_{\widehat{\mc{L}_U}}(P_U) = h_{\widehat{\omega_{\mc{X}'_{U'}/\mc{O}_{E_{X'}}}}}(P'_{U'}) = \sum_{\tau \in \phib} \paren*{h(\phib \sqcup \phib', \tau) + h(\conj{\phib} \sqcup \phib', \conj{\tau})}.\]
	
	
	Of note is that this formula does not depend on the choice of complementary partial CM-type $\phib'$, because the Shimura variety $X_U$ was defined independently of the choice of $\phib'$. We utilize this by summing over all possible complementary CM-type, which will express $h_{\widehat{\mc{L}_U}}(P_U)$ in terms of heights of full CM-types containing $\phib$ as well as nearby CM-types. The sum of heights of nearby CM-types is shown to be constant in \cite{Yua18}, and equal to the averaged height of all CM-types of $E$. Combining these two, we are able to express the height in terms of CM-types containing $\phib$ and an average of all possible CM-types.	
	
	\subsection{Structure of the Article}
	
	We first recall from \cite{Yua18} the decomposition of a Faltings height in Section \ref{sec:Faltings height}. We then describe three Shimura varieties that can be constructed from a quaternion algebra following \cite{Tia16} by describing their generic fiber in Section \ref{sec:Quaternionic Shimura Variety} and integral models in Section \ref{sec:Integral Model}. We then describe some line bundles on these Shimura varieties in terms of Lie algebras of certain $p$-divisible groups described in Section \ref{sec:p divisible group} and relate these Lie algebras, and finally define the Hodge bundle, in Section \ref{sec:Hodge bundle}. Finally, we prove our theorem for the height of partial CM-types in Section \ref{sec:Special Point} and compare our height with those introduced in \cite{Pil21} in Section \ref{sec:Andre Oort}.

    \subsection{Acknowledgements} We wish to thank Sebastian Eterovi\'c, Ananth Shankar, and Xinyi Yuan for their helpful discussions with us.
 
 \section{CM-types and Faltings Heights}\label{sec:Faltings height}
	\subsection{Faltings Height}
	
	We first define the Faltings height of an abelian variety. It will be defined as the degree of a metrized line bundle. Let $A$ be an abelian variety of dimension $g$ defined over a number field $K$ and let $\mc{A}$ be the N\'eron model over $\mc{O}_K$ and let the identity section be $s\colon \Spec \mc{O}_K \to \mc{A}$. Let $\Omega_{\mc{A}/\mc{O}_K}$ be the sheaf of relative differentials. The \emph{Hodge bundle} of $A$ is the vector bundle $\Omega(\mc{A}) \coloneqq s^*\Omega_{\mc{A}/\mc{O}_K}$ over $\mc{O}_K$. This is canonically isomorphic to the pushforward $\pi_* \Omega_{\mc{A}/\mc{O}_K}$, where $\pi\colon \mc{A} \to \mc{O}_K$ is the structure sheaf morphism.

	The Hodge bundle $\Omega(\mc{A})$ is a vector bundle over $\mc{O}_K$ of rank $g$ and taking the determinant $\omega(\mc{A}) \coloneqq \Omega(\mc{A})^{\wedge g}$ gives a line bundle over $\mc{O}_K$. To make this into a metrized line bundle, we need to define a norm for each archimedean place of $K$. We have that
	\[\omega(\mc{A}) \otimes_{\mc{O}_K} K \cong s^* \omega_{A/K} = H^0(A, \omega_{A/K}).\]
	For each archimedean place $v$ of $K$, we put the norm as
	\[\norm{\alpha}_v \coloneqq \abs*{\frac{1}{(2\pi)^g}\int_{A_v(\C)}\alpha \wedge \conj{\alpha}}^{\frac{1}{2}}\]
	for each $\alpha \in \omega(\mc{A}) \otimes_{\mc{O}_K} K_v \cong H^0(A_v, \omega_{A_v/K_v})$. In this way, we get a metrized line bundle $\widehat{\omega(\mc{A})} \coloneqq (\omega(\mc{A}), \norm{\cdot}_v)$.
	
	\begin{defn}
		The \emph{Faltings height} of the abelian variety $A/K$ is the Arakelov height
		\[h(A) \coloneqq \frac{1}{[K: \Q]} \widehat{\deg} \widehat{\omega(\mc{A})} = \frac{1}{[K:\Q]} \paren*{\log \abs{\omega(\mc{A})/(\mc{O}_K\cdot s)} - \sum_{\sigma \colon K \to \C}\log \|s\|_\sigma},\]
		for a choice of $s \in \omega(\mc{A}) \backslash \{0\}$. This is well defined and independent of the choice of $s$ by the product formula.
	\end{defn}
	
	If $A$ has semistable reduction over $K$, then the Faltings height is invariant under finite field extensions. In general, we define the stable Faltings height as the height after base change to a finite extension $K'/K$ such that $A$ has semistable reduction over $K'$. Such a $K'$ always exists.
	
	\subsection{CM-types}

	A \emph{CM-field extension} is an extension $E/F$ of number fields such that $F/\Q$ is a totally real field and $E/F$ is a quadratic totally imaginary extension. We say $E$ is a \emph{CM-field} and $F$ is its \emph{totally real subfield}.
		
	A \emph{(full) CM-type} is a subset $\Phi \subset \Hom(E, \C)$ such that $\Phi \sqcup \conj{\Phi} = \Hom(E, \C)$, where $\conj{\Phi} = \{\conj{\sigma} : \sigma \in \Phi\}$. A \emph{partial CM-type} is a subset $\phib \subset \Hom(E, \C)$ such that $\phib \cap \conj{\phib} = \emptyset$. We say that $\phib'$ is a \emph{complementary partial CM-type to $\phib$} if $\phib \sqcup \phib'$ is a CM-type.
		
	We say that a complex abelian variety $A$ has complex multiplication of type $(\mc{O}_E, \Phi)$ if there exists an embedding $\iota\colon \mc{O}_E \to \End(A)$ and an isomorphism $\Lie(A) \cong \C^g \overset{\Phi}{\cong} E \otimes_\Q \R$ of $\mc{O}_E$ modules.

	Let $E$ be a CM-field with degree $[E:\Q] = 2g$ and let $\Phi \subset \Hom(E, \C)$ be a CM-type. Let $A_\Phi$ be an abelian variety of CM-type $(\mc{O}_E, \Phi)$. Then, there is a number field $K$ over which $A_\Phi$ is defined and has a smooth projective integral model $\mc{A}/\mc{O}_K$. Colmez proved the following theorem
	
	\begin{thm}[{{\cite[Thm 0.3]{Col93}}}]
		The Faltings height $h(A_\Phi)$ depends only on the CM-type $(E, \Phi)$.
	\end{thm}

	We write $h(\Phi) \coloneqq h(A_\Phi)$. Colmez conjectured a formula about $h(\Phi)$ in terms of logarithmic derivatives of Artin L-functions related to $\Phi$. This conjecture has been proven when $E/\Q$ is an abelian extension by Obus and Colmez (see \cite{Obu13}) and when $F$ is a real quadratic field by Yang (see \cite{Yan10}). An averaged version was proven in \cite{Yua18}, and independently in \cite{And18}.
	
	\begin{thm}[{{\cite[Thm A]{And18}, \cite[Thm 1.1]{Yua18}}}]
		Suppose $E/F$ is an CM-extension and let $\chi\colon \mb{A}_F^\times \to \set{\pm 1}$ the character corresponding to this extension, and $L(s, \chi)$ the corresponding Artin L-function. Let $d_F$ be the absolute discriminant of $F$ and $d_{E/F}$ the norm of the relative discriminent of $E/F$. Then
		\[\frac{1}{2^g} \sum_\Phi h(\Phi) = - \frac{1}{2} \frac{L'(0, \chi)}{L(0, \chi)} - \frac{1}{4} \log (d_{E/F}d_F),\]
		where the sum on the left runs through the set of all CM-types of $E$.
	\end{thm}
	
	\subsection{Decomposition of Heights}
	
	We recall the results of \cite{Yua18} decomposing the Faltings height of a CM-type $\Phi$ into its constituent embeddings $\tau \in \Phi$. To decompose the height, we first decompose the Hodge bundle into its eigenspaces.
	
	Let $A$ be an abelian variety with complex multiplication of type $(\mc{O}_E, \Phi)$. We define
	\[\Omega(A)_\tau \coloneqq \Omega(A) \otimes_{E_\tau} \C,\]
	where $E_\tau$ acts on $\C$ through the projection $E \otimes_\Q \C \cong \prod_{\sigma \in \Hom(E, \C)} \C_\sigma \to \C_\tau$. This gives us a decomposition of the Hodge bundle as
	\[\Omega(A) \cong \bigoplus_{\tau\colon E \to \C} \Omega(A)_\tau \cong \bigoplus_{\tau \in \Phi} \Omega(A)_\tau.\]
	The latter isomorphism holds because $\Omega(A)_\tau = 0$ for $\tau \not \in \Phi$.
	
	Let $A^t$ be the dual abelian variety of $A$. Then, we have canonical isomorphisms
	\[\Omega(A^t) = \Lie(A^t)^\vee \cong H^1(A, \mc{O}_A)^\vee \cong H^{0, 1}(A)^\vee = \conj{\Omega(A)}^\vee,\]
	so that if $A$ is of CM-type $(\mc{O}_E, \Phi)$, then $A^t$ is of CM-type $(\mc{O}_E, \conj{\Phi})$. From this isomorphism, we also get a perfect Hermitian pairing $\Omega(A^t) \otimes \Omega(A) \to \C$.

	Just as before, we can decompose
	\[\Omega(A^t) \cong \bigoplus_{\tau \in \conj{\Phi}} \Omega(A^t)_\tau.\]
	The Hermitian pairing from before decomposes into a sum of orthogonal pairings $\Omega(A)_\tau \otimes \Omega(A^t)_{\conj{\tau}} \to \C$. Taking the determinant gives a Hermitian norm on the line bundle
	\[N(A, \tau) \coloneqq \det \Omega(A)_\tau \otimes \det \Omega(A^t)_{\conj{\tau}}.\]
	
	We can extend $N(A, \tau)$ to an integral model of $A$. If $\mc{A}$ is the N\'eron model over $\mc{O}_K$ as before, with $K$ including all embeddings of $E \to \clos{\Q}$, define
	\[\Omega(\mc{A})_\tau \coloneqq \Omega(\mc{A}) \otimes_{\mc{O}_K \otimes \mc{O}_E, \tau} \mc{O}_K\]
	for each $\tau\colon E \to K$. We define $\Omega(\mc{A}^t)_\tau$ analogously. For each archimedean place of $K$, we use the aforementioned Hermitian norm $\norm{\cdot}$ on the generic fiber of $\det \Omega(\mc{A})_\tau \otimes \det \Omega(\mc{A}^t)_{\conj{\tau}}$, and thus we get a metrized line bundle
	\[\widehat{\mc{N}(\mc{A}, \tau)} \coloneqq (\det \Omega(\mc{A})_\tau \otimes \det \Omega(\mc{A}^t)_{\conj{\tau}}, \norm{\cdot}).\]
	\begin{defn}
		If $A$ is an abelian variety of CM-type $(E, \Phi)$ and $\tau\colon E \to \C$, then the $\tau$-part of the Faltings height of $A$ is
		\[h(A, \tau) \coloneqq \frac{1}{2[K:\Q]} \widehat{\deg}\widehat{\mc{N}(\mc{A}, \tau)}.\]
	\end{defn}
	Note that if $\tau \not \in \Phi$, then $\mc{N}(\mc{A}, \tau) = 0$ and so the height contribution is $0$ as well.
	
	Just as with the Faltings height, this $\tau$-component is independent of the abelian variety itself. Thus, we will write $h(\Phi, \tau)$ for $h(A, \tau)$.
	
	\begin{thm}[{{\cite[Thm 2.2]{Yua18}}}]\label{thm:fullvscomponent}
		If $A$ has CM of type $(\mc{O}_E, \Phi)$, the height $h(A, \tau)$ depends only on the pair $(\Phi, \tau)$.
	\end{thm}
	
	We call a pair of CM-types $(\Phi_1, \Phi_2)$ \emph{nearby} if $\abs{\Phi_1 \cap \Phi_2} = g - 1$. Let $\tau_i = \Phi_i \backslash (\Phi_1 \cap \Phi_2)$ be the place where they differ. Then, the sum of the $\tau_i$-components of $h(\Phi_i)$ is independent of the choice of nearby CM-type.
	
	\begin{thm}[{{\cite[Thm. 2.7]{Yua18}}}]\label{thm:nearby CM-type}
	    The quantity $h(\Phi_1, \tau_1) + h(\Phi_2, \tau_2)$ is independent of the choice of nearby CM-type $(\Phi_1, \Phi_2)$.
	\end{thm}
	
	Finally, we compare $h(\Phi)$ with its constituents $h(\Phi, \tau)$.
 
        \begin{defn}\label{def:CM Type Determinant}
            Let $\Psi \subset \Hom(E, \C)$ be any subset, not necessarily a (partial) CM-type. The reflex field $E_\Psi \subset E^{\Gal}$ is the subfield of the Galois closure of $E$ fixed by all automorphisms that fix $\Psi$. The trace map $\Tr_\Psi \colon E \to E_\Psi$ is given by $\Tr_\Psi(x) = \sum_{\tau \in \Psi} \tau(x)$.

            We can decompose $E_\Psi \otimes_\Q E \cong \widetilde{E_\Psi} \times \widetilde{E_{\comp{\Psi}}}$ where the trace of the action of $E$ on $\widetilde{E_\Psi}$ is $\Tr_\Psi$ and the trace of the action on $\widetilde{E_{\comp{\Psi}}}$ is $\Tr_{\comp{\Psi}}$. Let $\mf{d}_\Psi$ be the relative discriminant of the image of $\mc{O}_{E_\Psi} \otimes_\Z \mc{O}_E$ in $\widetilde{E_\Psi}$ over $\mc{O}_{E_\Psi}$, and let $d_\Psi$ be the positive generator of the $N_{E_\Psi/\Q}(\mf{d}_\Psi)$.
        \end{defn}
        
	\begin{thm}[{{\cite[Thm 2.3]{Yua18}}}]
		\[h(\Phi) - \sum_{\tau \in \Phi} h(\Phi, \tau) = \frac{-1}{4[E_\Phi:\Q]} \log (d_\Phi d_{\conj{\Phi}}).\]
	\end{thm}
	
	\section{Quaternionic Shimura Varieties}\label{sec:Quaternionic Shimura Variety}
	
	We fix a totally real field $F/\Q$ of degree $g$. Let $\Sigma \subset \Hom(F, \R)$ be the subset of places of $F$ and let $B/F$ be a quaternion algebra over $F$ that is split at infinity precisely at $\Sigma$, which means that
	\[B \otimes_\Q \R \cong \prod_{\tau \in \Sigma} M_2(\R)_\tau \oplus \prod_{\sigma \not \in \Sigma} \mb{H}_\sigma.\]
	From this quaternion algebra, we will construct three related quaternionic Shimura varieties and relate them.

	We are primarily interested in the Shimura variety $X$ associated with the group $G = \Res_{F/\Q} B^\times$. However, this Shimura datum does not parametrize abelian varieties and is of abelian type. We will follow the approach of \cite{Car86} by finding a unitary Shimura datum $G'$ that has the same derived group as $G$, and is of PEL type which will give us a nice description of the integral models of $X'$ in terms of abelian varieties. Then following \cite{Kis10, Kis18}, we will give an integral model for $X$ by transporting the connected components of $X'$.
	
	We start with the primary Shimura variety of study, the one associated to the group $G = \Res_{F/\Q} (B^\times)$. Let $h\colon \C^\times \to G(\R)$ be a cocharacter defined by
	\[h(a + bi) = \paren*{\prod_{\tau \in \Sigma} \begin{pmatrix}
			a & b\\-b & a
		\end{pmatrix}_\tau, \prod_{\sigma \not \in \Sigma} 1_\sigma} \in \prod_{\tau \in \Sigma} M_2(\R)_\tau \times \prod_{\sigma \not \in \Sigma} \mb{H}_\sigma.\]
        We can identify the $G(\R)$-conjugacy class of $h$ with $(\mc{H}^{\pm})^{\abs \Sigma}$, where $\mc{H}^\pm \coloneqq \C \backslash \R$, by sending $ghg^{-1} \mapsto \prod_{\tau \in \Sigma} g_\tau(i_\tau)$, where $g_\tau \in \GL_2(\R)_\tau$ is the $\tau$ component of $g$, and acts on $i$ through M\"obius transformations. From the Shimura datum $(G, (\mc{H}^\pm)^{\Sigma})$, we get a Shimura variety $X_U$ for each open compact subgroup $U \subset G(\mb{A}_f)$ that has a complex uniformization given by
	\[X_U(\C) = G(\Q) \backslash (\mc{H}^\pm)^{ \Sigma} \times G(\mb{A}_f)/U.\]
	The reflex field $E_X \coloneqq E(G, (\mc{H}^\pm)^{\Sigma})$ of $X$ is the subfield of $\C$ fixed by the automorphisms of $\C$ that fix $\Sigma \subset \Hom(F, \R)$. The Shimura variety $X_U$ has a canonical model over $E_X$ whose complex points have the above uniformization (see \cite{Mil90}).
	
	Let $N_{B/F}\colon B^\times \to F^\times$ be the reduced norm on $B$. The derived group is $G^{\der} = \ker\paren*{N_{B/F}\colon \Res_{F/\Q} B^\times \to \Res_{F/\Q} \mb{G}_m}$, the elements of $B$ with norm $1 \in F$, and its adjoint group is $\Res_{F/\Q} B^\times/F^\times$.
	
	We now introduce two auxiliary Shimura data that have the same derived group and the same adjoint group as $G$. Thus, their associated Shimura varieties have isomorphic connected components to $X$. Let $E/F$ be a CM extension such that there is an embedding $E \hookrightarrow B$. Let
	\[G'' \coloneqq \Res_{F/\Q} (B^\times \times E^\times)/F^\times,\]
	where $F^\times \hookrightarrow B^\times \times E^\times$ by $a \mapsto (a, a^{-1})$. Let $\Phi\colon E\otimes_\Q \R \overset{\sim}{\to} \C^g$ be a full CM-type of $E$. Split $\Phi$ into partial CM-types by
	\[\phib = \{\sigma \in \Phi : \sigma|_F \in \Sigma\},\]
	and
	\[\phib' = \{\sigma \in \Phi : \sigma|_F \not \in \Sigma\},\]
	so that $\phib$ and $\phib'$ are complementary partial CM-types. These partial CM-types give maps $\phib\colon E \to \C^{\Sigma}$ and $\phib'\colon E \to \C^{\comp{\Sigma}}$. Identify $E \otimes_\Q \R$ with $\C^g$ through the CM-type $\Phi$. Define the cocharacter $h_E\colon \C^\times \to E \otimes_\Q \R$ to be
	\[h_E(z) = (\phib(1), \phib'(z)) \in \C^g.\]
	We can now define the cocharacter $h'' \colon \C^\times \to G''(\R)$ as the composition of the map $(h(z), h_E(z))$ with the quotient by $F^\times$-action map. As before, the $G''(\R)$-conjugacy class of $h''$ can be identified with $(\mc{H}^{\pm})^\Sigma$.
	
	There exists a well defined norm $\nu\colon G'' \to \Res_{F/\Q} \mb{G}_m$ given by mapping
 \[(b, e) \mapsto N_{B/F}(b)N_{E/F}(e).\]
    We use this norm to define an algebraic subgroup
	\[G' \coloneqq G'' \times_{\Res_{F/\Q} \mb{G}_m}\mb{G}_m,\]
	which consists of elements of $G''$ whose norm lies in $\Q^\times \subset F^\times$. For our cocharacter $h''$, its norm is $\nu(h''(a + bi)) = a^2 + b^2 \in \R^\times \subset (F \otimes_\Q \R)^\times$. Hence $h''$ factors through a map $h' \colon \C^\times \to G'(\R)$. The $G'(\R)$ conjugacy class of $h'$ can be identified with $(\mc{H}^\pm)^{\abs\Sigma}$ as well. For open and compact subgroups $U' \subset G'(\mb{A}_f)$ and $U'' \subset G''(\mb{A}_f)$, we get Shimura varieties $X'_{U'}$ and $X''_{U''}$ with complex uniformizations
	\[X'_{U'}(\C) = G'(\Q) \backslash (\mc{H}^\pm)^{\Sigma} \times G'(\mb{A}_f) / U',\]
	and
	\[X''_{U''}(\C) = G''(\Q) \backslash (\mc{H}^\pm)^{\Sigma} \times G''(\mb{A}_f) / U''.\]
	The reflex fields of these Shimura varieties $E_{X'} \coloneqq E(G',(\mc{H}^\pm)^{\Sigma})$ and $E_{X''} \coloneqq E(G'', (\mc{H}^\pm)^{\Sigma})$ are the same, and equal to the subfield of $\C$ fixed by all automorphisms of $\C$ fixing $\phib' \subset \Hom(E, \C)$. If an automorphism of $\C$ fixes $\phib'$, then it fixes $\Sigma$ as well. Therefore, the reflex field $E_X$ is a subfield of $E_{X'} = E_{X''}$.

	We now describe the abelian varieties which $X'$ parametrizes. Let $V = B$ be viewed as a $\Q$-vector space with a natural left action by $E$, and choose $\gamma \in E \subset B$ so that $\conj{\gamma} = -\gamma$. We define a similitude $\psi\colon V \times V \to \Q$ by
	\[\psi(v, w) = \Tr_{F/\Q} \Tr_{B/F}(\gamma v \conj{w}),\]
	where $\conj{w}$ is the canonical involution on $B$. This is a nondegenerate alternating form, and $\psi(ev, w) = \psi'(v, e^*w)$ for all $v, w \in V$ and $e \in E$, where the involution $e^* = \conj{e}$ is just complex conjugation on $E$. We define a left action of $(B \times E)^\times/F^\times$ on $V$ by setting $(b, e) \cdot v = ev\conj{b}$. In this way, we can identify $G'$ with $E$-linear automorphisms of $V$ with rational norm
	\[G' = \{g \in \GL_{E}(V) : \psi(gv, gw) = \nu(g) \cdot \psi(v, w)\text{ for some }\nu(g) \in \mb{G}_m\}.\]
	
	The action of $\C$ on $V_\R$ through the morphism $h'$ induces a Hodge structure on $V$ of weight $1$. We can choose $\gamma$ such that $\psi$ induces a polarization satisfying $\psi(v, h'(i)v) \ge 0$ for all $v \in V_\R$, and hence $(V, \psi)$ is a symplectic $(E, *)$-module.
	
	Thus, by \cite[Thm 8.17]{Mil05}, the pair $(G', (\mc{H}^\pm)^{\Sigma})$ is PEL Shimura datum and $X'_{U'}$, for $U'$ small enough, represents the functor, for any test scheme $S$ over $E_{X'}$, whose $S$-points are isomorphism classes of quadruples $(A, \iota, \theta, \kappa U')$ where
	\begin{enumerate}
		\item $A/S$ is an abelian scheme of relative dimension $2g$;
		\item $\iota\colon E \to \End(A/S)\otimes_\Z \Q$ is an injection such that the action of $\iota(E)$ on $\Lie(A/S)$ has trace given by
		\[\Tr(\ell, \Lie(A/S)) = \Tr_{\phib \sqcup \phib'}(\ell) + \Tr_{\conj{\phib} \sqcup \phib'}(\ell)\]
		for all $\ell \in E$, where for a CM-type $(E, \Phi)$ the trace map $\Tr_\Phi\colon E \to \C$ is defined as
		\[\Tr_\Phi(e) = \sum_{\sigma \in \Phi} \sigma(e);\]
		\item $\theta\colon A \to A^t$ is a polarization whose Rosati involution on $\End(A/S)_\Q$ induces the involution $\gamma \mapsto \gamma^*$ on $E$;
		\item and $\kappa\colon H_1(A, \mb{A}_f) \simeq V_{\mb{A}_f}$ is an isomorphism of $\mb{A}_{E, f}$-modules that respects the bilinear forms on both factors, up to an element in $\mb{A}_f^\times$.
	\end{enumerate}

	
	\section{Integral Models}\label{sec:Integral Model}
	
	To construct integral models for these Shimura varieties, we first use the PEL structure of $X'$ to get an integral model $\mc{X}'$ which parametrizes abelian schemes. Then we will transfer the integral model of $X'$ to construct integral models for $X$ and $X''$, as done in \cite{Kis10, Kis18}.
	
	\subsection{PEL Type \texorpdfstring{$\mc{X}'$}{X'}}
	
	We construct $\mc{X}'$ following \cite{Rap96, Pap13}. Let $p \in \Z$ be a prime number. Let $\mf{p}$ be a prime of $F$ lying above $p$. Set $\mc{O}_{E, \mf{p}} = \mc{O}_{E} \otimes_{\mc{O}_F} \mc{O}_{F, \mf{p}}$.
    \begin{itemize}
        \item If $B$ is unramified at $\mf{p}$, then $B_\mf{p} \cong M_2(F_\mf{p})$. Choose an isomorphism such that $\mc{O}_{E, \mf{p}} \subset M_2(\mc{O}_{F, \mf{p}})$ and set $\Lambda_\mf{p} = M_2(\mc{O}_{F, \mf{p}})$.

        \item If $B$ is ramified at $\mf{p}$, then $B_\mf{p}$ is a division algebra over $F_\mf{p}$ and there is a unique choice of a maximal order $\mc{O}_{B, \mf{p}}$, which must contain $\mc{O}_{E, \mf{p}}$. We set $\Lambda_\mf{p} = \mc{O}_{B, \mf{p}}$.
    \end{itemize}
    From this choice of $\mc{O}_{E, \mf{p}}$-lattice $\Lambda_\mf{p}$, construct a chain of lattices by taking
    \[\mc{L}_{\mf{p}} = \{\cdots \subset \omega_{\mf{q}}\Lambda_{\mf{p}} \subset \Lambda_{\mf{p}} \subset \omega_{\mf{q}}^{-1} \Lambda_{\mf{p}} \subset \cdots\},\]
    where $\omega_{\mf{q}}$ is a uniformizer of $E_\mf{q}$, taken to be a uniformizer of $F_\mf{p}$ if $\mf{q}$ is unramified over $\mf{p}$. From these chains, we can construct a multichain $\mc{L}_p$ of $\mc{O}_{E} \otimes_\Z \Z_p$-lattices which consist of all lattices $\Lambda_p$ which can be written as
    \[\Lambda_p = \oplus_{\mf{p} \mid p} \Lambda_\mf{p}, \quad \Lambda_\mf{p} \in \mc{L}_{\mf{p}}.\]

    We now require that $E$ is ramified above all finite primes $\mf{p}$ of $F$ that also ramify in $B$. Let $\delta_{\mf{p}/p} \in F_\mf{p}$ be a generator for the different ideal $\mf{d}_{F_\mf{p}/\Q_p}$ of $F_\mf{p}/\Q_p$.
    \begin{lem}
        We can choose $\gamma \in E^\times$ such that
        \begin{itemize}
            \item $\gamma = -\conj{\gamma}$;
            \item $\gamma \in \delta_{\mf{p}/p}^{-1}\mc{O}_{E, \mf{p}}^\times$;
            \item and $\psi(v, h'(i)v) > 0$ for all $v \in V_\R \backslash \{0\}$.
        \end{itemize}

        Moreover, under this choice of $\gamma$, the multichain of $\mc{O}_E \otimes_\Z \Z_p$-lattices $\mc{L}_p$ is self-dual with respect to the $\Q_p$-linear alternating form $\psi_p \colon V_p \times V_p \to \Q_p$.
    \end{lem}
    \begin{proof}
        The anti-symmetric elements of $E$ are dense in the anti-symmetric elements of $(E \otimes_\Q \Q_p) \oplus (E \otimes_\Q \R)$ and the conditions given are all open and non-empty, meaning that we can find such a $\gamma$.

        To show that $\mc{L}_p$ is a self-dual multichain, it suffices to look locally at each prime $\mf{p}$ of $F$. The alternating form $\psi_p$ is the sum over all primes $\mf{p}$ of
        \[\psi_\mf{p}(v, w) = \Tr_{F_\mf{p}/\Q_p} \Tr_{B_\mf{p}/F_\mf{p}}(\gamma v \conj{w}).\]
        Thus, the dual of the lattice $\Lambda_\mf{p}$ with respect to $\psi_\mf{p}$ is
        \[\Lambda_{\mf{p}}^\vee = \{w \in V_\mf{p} : \psi_\mf{p}(v, w) \in \Z_p \forall v \in \Lambda_\mf{p}\} = \{w \in V_\mf{p} : \Tr_{B_\mf{p}/F_\mf{p}}(\gamma v \conj{w}) \in \delta_{\mf{p}/p}^{-1}\mc{O}_{F, \mf{p}}\}.\]

        It suffice to take $\delta = \delta_{\mf{p}/p}^{-1}$. First assume that $\mf{p}$ is unramified in $B$. Under the isomorphism $B_\mf{p} \cong M_2(F_\mf{p})$, the involution on $B_\mf{p}$ is given by $\conj{\begin{psmallmatrix}
            a & b\\c & d
        \end{psmallmatrix}} = \begin{psmallmatrix}
            d & -b \\ -c & a
        \end{psmallmatrix}$ and the similitude is
        \[\Tr_{B_\mf{p}/F_\mf{p}}\paren*{\delta\begin{psmallmatrix}a & b\\c& d\end{psmallmatrix}\conj{\begin{psmallmatrix}a'&b'\\c'&d'\end{psmallmatrix}}} = \delta_{\mf{p}/p}^{-1}(ad' + bc' + a'd + b'c).\]
        From this, we see that $\Lambda_\mf{p}^\vee = \Lambda_\mf{p} = M_2(\mc{O}_{F, \mf{p}})$.

        If $\mf{p}$ is ramified in $B$, we required that $E$ is also ramified at $\mf{p}$. So, we can find an element $j \in B_\mf{p}$ such that $j^2 \in \mc{O}_{F, \mf{p}}^\times$ and $je = \conj{e}j$ for all $e \in E_\mf{p}$. For this choice of $j$, the unique maximal order can be written as $\mc{O}_{B, \mf{p}} = \mc{O}_{E, \mf{p}} + \mc{O}_{E, \mf{p}}j$. For $a + bj \in E_\mf{p} + E_\mf{p}j = B_\mf{p}$, the trace is $\Tr_{B_\mf{p}/F_\mf{p}} (a + bj) = \Tr_{E_\mf{p}/F_\mf{p}} (a)$. We thus have
        \[\Tr_{B_\mf{p}/F_\mf{p}}\paren*{\gamma(a + bj)(\conj{a' + b'j})} = \delta_{\mf{p}/p}^{-1} \Tr_{E_\mf{p}/F_\mf{p}} (a\conj{a'} - b\conj{b'}j^2).\]
        From this, we get that $\Lambda_{\mf{p}}^\vee = \omega_\mf{q}^{-1}\Lambda_\mf{p} \in \mc{L}_\mf{p}$.
    \end{proof}


    Let $U'_p = U'_p(0) \subset G'(\Q_p)$ be the parahoric subgroup of elements fixing the multichain $\mc{L}_p$. Let $U'^p\subset G'(\mb{A}_f^p)$ and set $U' \coloneqq U'_pU'^p \subset G'(\mb{A}_f)$. Choose a place $v'$ of $\mc{O}_{E_{X'}}$ lying above $p$. From the integral data of $(\mc{L}_p, U'_p)$, consider the functor $\mc{F}'_{U'_pU'^p}$ which associates to a test scheme $\mc{S}$ over $\Spec \mc{O}_{E_{X'}, v'}$ the set of isomorphism classes of quadruples $(\mc{A}, \iota, \theta, \kappa U'^p)$ where:
    
    \begin{enumerate}
		\item $\mc{A}/\mc{S}$ is an abelian scheme of relative dimension $2g$ up to isogeny of order prime to $p$;
		\item $\iota\colon \mc{O}_{E} \otimes_\Z \Z_{(p)} \to \End(\mc{A}/\mc{S}) \otimes_\Z \Z_{(p)}$ is a homomorphism satisfying the following Kottwitz condition. There is an identity of polynomial functions
            \[\det_{\mc{O}_S} (\iota(\ell); \Lie \mc{A}/S) = \prod_{\phi \in \phib} \phi(\ell)\conj{\phi(\ell)} \prod_{\phi' \in \phib'} \phi'(\ell)^2;\]
		\item $\theta\colon \mc{A} \to \mc{A}^t$ is a principle polarization whose Rosati involution on $\End(\mc{A}/S) \otimes \Z_{(p)}$ induces complex conjugation on $\mc{O}_{E, (p)}$;
		\item and $\kappa\colon H_1(\mc{A}, \mb{A}_f^p) \simeq V_{\mb{A}^p_f}$ is an isomorphism of skew $\mc{O}_{E, (p)}\otimes \mb{A}^p_f$-modules that respects the bilinear forms up to a constant in $(\mb{A}_f^p)^\times$.
	\end{enumerate}
	
	\begin{thm}
	    If $U'^p$ is sufficiently small, then the functor $\mc{F}'_{U'}$ is represented by a quasi-projective scheme $\mc{M}_{U'}$ over $\mc{O}_{E_{X'}, v'}$ whose generic fiber is $X'_{U'}$. Moreover, we have:
	    \begin{enumerate}
	        \item If $p$ is unramified in $E$, the scheme $\mc{M}_{U'}$ is smooth over $\mc{O}_{E_{X'}, v'}$;
	        \item The $p$-adic completion of $\mc{M}_{U'}$ along the basic locus has a $p$-adic uniformization by a Rapoport--Zink space 
	    \end{enumerate}
	\end{thm}
	\begin{proof}
	    This is the PEL moduli problem studied by \cite{Kot92} and \cite{Rap96}. The first case is covered by \cite[Sec. 5]{Kot92} and the second case is covered by \cite[Thm 6.50]{Rap96}.
	\end{proof}

    If $p$ is unramified in $E$, we set our integral model $\mc{X}_{U'}$ to be $\mc{M}_{U'}$. If $p$ is ramified in $E$, then $\mc{M}_{U'}$ is not necessarily flat. We explain how to construct a flat model following \cite{Pap13}. We first construct the corresponding local model for $\mc{F}'_{U'}$ following \cite[Def. 3.27]{Rap96}. Let $\mb{M}^{\mathrm{naive}}$ be the functor which associates to a locally Noetherian scheme $\mc{S}$ over $\Spec \mc{O}_{E_{X'}, v'}$ the set of $\mc{O}_{E, p} \otimes_{\Z_p} \mc{O}_S$-submodules $t_\Lambda \subset \Lambda_p \otimes_{\Z_p} \mc{O}_S$ such that
    \begin{enumerate}
        \item $t_\Lambda$ is a finite, locally free $\mc{O}_S$-module;
        \item For all $\ell \in \mc{O}_{E, p}$, there is an identity of polynomial functions
        \[\det_{\mc{O}_S}(\ell; t_\Lambda) =  \prod_{\phi \in \phib} \phi(\ell)\conj{\phi(\ell)} \prod_{\phi' \in \phib'} \phi'(\ell)^2;\]
        \item and $t_\Lambda$ is totally isotropic under the nondegenerate alternating pairing 
        \[\psi_{p, \mc{O}_S}\colon (\Lambda_p \otimes_{\Z_p} \mc{O}_S) \times (\Lambda_p \otimes_{\Z_p} \mc{O}_S) \to \mc{O}_S.\]
    \end{enumerate}
    This functor is represented by a closed subscheme of a Grassmannian. Let $\mc{P}/\Spec \Z_p$ be the group scheme whose $S$ points are $\Aut(\mc{L} \otimes_{\Z_p} \mc{O}_S)$ automorphisms of the multichain $\mc{L}$ that respect the similitude $\psi_p$. Then by \cite[3.30]{Rap96}, there is a smooth morphism of algebraic stacks of relative dimension $\dim G' = 5g$.
    \[\mc{M}_{U'} \to \left[\mb{M}^{\mathrm{naive}}/\mc{P}_{\mc{O}_{E_{X'}, v'}}\right]\]
    From this, we see that $\mb{M}^{\mathrm{naive}}$ controls the structure of $\mc{M}_{U'}$. While conjectured to be flat, it has since been shown by \cite{Pap00} that $\mb{M}^{\mathrm{naive}}$ is not necessarily flat when $p$ is ramified in the PEL datum. To remedy this, take $\mb{M}^{\mathrm{loc}}$ to be the flat scheme theoretic closure of $\mb{M}^{\mathrm{naive}} \otimes_{\mc{O}_{E_{X'}, v'}} E_{X', v'}$ in $\mb{M}^{\mathrm{naive}}$.

    \begin{prop}[{\cite[Thm. 9.1]{Pap13}}]
        The scheme $\mb{M}^{\mathrm{loc}}$ is normal and Cohen--Macaulay with reduced special fiber. It also admits an action by $\mc{P}_{\mc{O}_{E_{X'}, v'}}$ such that the natural inclusion $\mb{M}^{\mathrm{loc}} \to \mb{M}^{\mathrm{naive}}$ is $\mc{P}_{\mc{O}_{E_{X'}, v'}}$-equivariant.
    \end{prop}
    
    With this flat local model, we can define a flat integral model $\mc{X}'_{U'}$ for $X'_{U'}$ by pulling back $\mc{M}_{U'}$ to $\mb{M}^{\mathrm{loc}}$ to get the following cartesian diagram.
    \[\begin{tikzcd}
    \mc{X}'_{U'} \arrow[r] \arrow[d] &\mc{M}_{U'}  \arrow[d] \\ \left[\mb{M}^{\mathrm{loc}}/\mc{P}_{\mc{O}_{E_{X'}, v'}}\right] \arrow[r] & \left[\mb{M}^{\mathrm{naive}}/\mc{P}_{\mc{O}_{E_{X'}, v'}}\right]
    \end{tikzcd}\]
    The schemes $\mc{X}'_{U'}$ and $\mc{M}_{U'}$ have generic fiber equal to $X'_{U'}$ and since $\mb{M}^{\mathrm{loc}}$ is flat, our integral model $\mc{X}'_{U'}$ of $X'_{U'}$ is flat as well.

    We have defined integral models $\mc{X}'_{U'}$ when $U' = U'_pU'^p$ where $U'_p = U'_p(0)$ is maximally parahoric and $U'^p$ is sufficiently small. We want to define an integral model $\mc{X}'_{U'}$ when $U' = \prod_p U'_p(0)$ is maximal at all primes. We now show how to construct these integral models when $U'^p$ is big.

    Suppose that $p$ is unramified in $E$ so that $\mc{X'}_{U'}$ is smooth. For integral $m \ge 0$, let $U_p(m)$ denote
    \[U_p'(m) \coloneqq \set{g \in G'(\Q_p) : g\Lambda_p = \Lambda_p, g|_{p^{-m}\Lambda_p/\Lambda_p} \equiv 1}\]
    the subgroup of $U'_p$ that acts as the identity on $\Lambda_p/p^m\Lambda_p$. The nondegenerate alternating form $\psi_p$ on $\Lambda_p$ gives rise to a nondegenerate $*$-hermitian alternating form
    \[\inner{,}_{p, m}\colon p^{-m}\Lambda_p/\Lambda_p \times p^{-m}\Lambda_p/\Lambda_p \to p^{-m}\Z_p/\Z_p\]
    given by
    \[\inner{x, y}_{p, m} = \psi_p(p^mx, y).\]
    Then for $U'^p$ small enough, Mantovan (see \cite{Man05}) defines an integral model of $\mc{X}'_{U'_p(m)U'^p}$ over $\Spec \mc{O}_{E_{X'}, v'}$ by using the notion of a full set of sections. Let $\mc{F}'_{U'_p(m)U'^p}$ be the functor over $\mc{F}'_{U'_pU'^p} = \mc{F}'_{U'_p(0)U'^p}$ which associates to a locally Noetherian scheme $S$ over $\mc{O}_{E_{X'}, v'}$ the set of isomorphism classes of data $(\mc{A}, \iota, \theta, \kappa, \alpha)$ where $(\mc{A}, \iota, \theta, \kappa)$ are as in the functor $\mc{F}'_{U'_p(0)U'^p}$ and
    \[\alpha\colon p^{-m}\Lambda_p/\Lambda_p \to \mc{A}[p^m](S)\]
    is an $\mc{O}_{E, p}$-linear homomorphism such that $\set{\alpha(x) : x \in p^{-m}\Lambda/\Lambda}$ is a full set of sections of $\mc{A}[p^m]$ and $\alpha$ maps the pairing $\inner{,}_{p, m}$ to the Weil pairing on $\mc{A}[p^m]$, up to a scalar multiple in $(\Z/p^m\Z)^\times$.

    \begin{thm}[{\cite[Prop. 15]{Man05}}]
        The functor $\mc{F}'_{U'_p(m)U'^p}$ is represented by a smooth scheme $\mc{X}_{U'_p(m)U'^p}$ over $\mc{O}_{E_{X'}, v'}$.
    \end{thm}

    To make the notion of $U'^p$ small enough explicit, fix a lattice $\Lambda \subset V$ over $\Z$ such that $\Lambda \otimes_{\Z} \Z_p \cong \Lambda_p$. For $N \in \N$, let
    \[U'(N) \coloneqq \set{g \in G'(\mb{A}_f) : g|_{\Lambda/N\Lambda} \equiv 1}.\]
    
    We now remove our restriction that the moduli problem above $p$ is unramified.

    \begin{prop}
        If $U'_p(m)U'^p \subset U'(N)$ is a normal subgroup such that $N \ge 3$ and $m = 0$ when $p$ is ramified in $E_{X'}$, then the functor $\mc{F}'_{U'_p(m)U'^p}$ is representable by a normal scheme with reduced special fiber. Moreover, if $p$ is unramified in $E_{X'}$, then the functor $\mc{F}'_{U'_p(m)U'^p}$ is representable by a smooth scheme.
    \end{prop}
    \begin{proof}
        Choose a normal subgroup $U'^p_0 \subset U'^p$ sufficiently small so that the functor $\mc{F}'_{U'_p(m)U'^p_0}$ is represented by the scheme $\mc{X}'_{U'_p(m)U'^p_0}$. There is an action of $U'^p \subset U'(N)^p$ on this scheme and it suffices to show that $U'(N)^p$ acts freely on $\mc{X}'_{U'_p(m)U'^p_0}$. A point $x \in \mc{X}'_{U'_p(m)U'^p_0}$ corresponds to a quintuple $(\mc{A}, \iota, \theta, \kappa, \alpha)$, where the full set of sections $\alpha$ is trivial when $m = 0$. Suppose that $g \in U'(N)^p$ fixes $x$. We may choose our $\mc{A}$ and $\kappa$ so that $H_1(\mc{A}, \mb{A}_f^p) \cong \Lambda_{\mb{A}_f^p}$ and $\kappa$ induces an isomorphism between the two. The element $g$ acts by sending $(\mc{A}, \iota, \theta, \kappa, \alpha)g = (\mc{A}, \iota, \theta, g^{-1} \circ \kappa, \alpha)$. Thus, there exist some isomorphism $f$ of $\mc{A}$ and element $g' \in U'^p_0$ such that $(gg')^{-1} \circ \kappa = \kappa \circ f_*\colon H_1(\mc{A}, \mb{A}_f^p) \to \Lambda_{\mb{A}_f^p}$. Now since $gg' \in U'(N)^p$ acts on the identity on $\Lambda/N\Lambda$, we get that $f_*$ must act as the identity on $\mc{A}[N]$, meaning that $f$ is the identity since $N \ge 3$. Thus $gg' = 1$ and $g^{-1}\circ \kappa$ is in the same $U'^p_0$ orbit as $\kappa$.
    \end{proof}

	We let $\mc{X}'_{U'}$ be the scheme representing $\mc{F}'_{U'}$ whenever $U' = \prod_p U'_p \subset U'(N)$ for $N \ge 3$. This gives us a flat and normal integral model for $X'_{U'}$ over $\mc{O}_{E_{X'}}$.
 
    \subsection{Transferring Integral Models}
    
    Now we can use the integral model for $X'$ to get integral models for $X$ and $X''$, as done in \cite{Kis10, Kis18}, by extending the adjoint group $G^{\ad}$-action on the neutral component of $\mc{X}'$. We briefly recall how this is done because we will use the same idea to transfer the $p$-divisible group on $\mc{X}'$ to $p$-divisible groups over $\mc{X}$ and $\mc{X}''$. Set $U_p(m) \coloneqq (1 + p^m\mc{O}_{B, p})^\times$ and $U''_p \coloneqq U''_p(m) \coloneqq (1 + p(\mc{O}_{B, p}\times \mc{O}_{E, p}))^\times \subset G''(\Q_p)$. Then $U_p \coloneqq U_p(0)$ and $U''_p \coloneqq U''_p(0)$ are the $\Z_p$-points of parahoric subgroups over $\Z_{(p)}$ which fix the lattices $\mc{O}_{B, p}$ and $\mc{O}_{B, p} \times \mc{O}_{E, p}$ respectively, and over the generic fiber they are isomorphic to $G$ and $G''$. Denote these models $G_p, G''_p$ so that $G_p(\Z_p) = U_p$ and $G''_p(\Z_p) = U''_p$. For $S = X, X', X''$, let $U_{S, p} = U_p, U'_p, U''_p$ and $U_{S}^p = U^p, U'^p, U''^p$ and $G_S = G, G', G''$ and $G_{S, p} = G_p, G'_p, G''_p$ respectively. Let $Z_S$ be the center of $G_S$. For each choice of $S$, take the limit over all choices of $U_{S}^p$ to get
    \[S_{U_{S, p}} = \varprojlim_{U_{S}^p} S_{U_{S, p}U_{S}^p} = G_S(\Q) \backslash (\mc{H}^{\pm})^\Sigma \times G_S(\mb{A}_f)/U_{S, p},\]
    and let the entire projective limit be
    \[S = \varprojlim_{U_{S, p}U_S^p} S_{U_{S, p}U_S^p} = G_S(\Q) \backslash (\mc{H}^\pm)^\Sigma \times G_S(\mb{A}_f).\]

    We recall the star product notation of \cite{Kis10, Kis18}. Suppose that a group $\Delta$ acts on a group $H$ and suppose that $\Gamma \subset H$ is stable under the action of $\Delta$. Let $\Delta$ act on itself by left conjugation, and suppose there is a group homomorphism $\phi\colon \Gamma \to \Delta$ that respects $\Delta$-action. We also impose for all $\gamma \in \Gamma$ that the $\phi(\gamma)$-action on $H$ is by left conjugation by $\gamma$. Then, the subgroup $\{(\gamma, \phi(\gamma)^{-1}) : \gamma \in \Gamma\}$ is a normal subgroup of $H \rtimes \Delta$ and we let $H *_\Gamma \Delta$ denote the quotient.
    
    Let $G_S^{\ad}(\R)^{+}$ be the neutral component and let $G_S(\R)_+$ be the preimage of $G_S^{\ad}(\R)^{+}$ under the map $G_S(\R) \to G_S^{\ad}(\R)$. Let $G_S(\Q)_+ = G_S(\R)_+ \cap G_S(\Q)$ and let $G_{S, p}(\Z_{(p)})_+ = G_{S, p}(\Z_{(p)}) \cap G_S(\Q)_+$. Let $G_{S, p}^{\ad}(\Z_{(p)})^+ = G_{S, p}^{\ad}(\Z_{(p)}) \cap G_S^{\ad}(\R)^+$. There is a natural right action of $G_S(\mb{A}_f)$ on $S$ by right multiplication on the $G_S(\mb{A}_f)$ factor, under which $Z_S(\Q)$ acts trivially. There is also a right action of $G^{\ad}_S(\Q)^+$ on $S$ where $\gamma \in G_S^{\ad}(\Q)^+$ acts on a representative $[x, g]$ by $[x, g]\gamma = [\gamma^{-1}x, \gamma^{-1}g\gamma]$. Let $Z_S(\Q)^-$ be the closure of $Z_S(\Q)$ in $G_S(\mb{A}_f)$. We thus get an action of
    \[G_S(\mb{A}_f)/Z_S(\Q)^- \rtimes G_S^{\ad}(\Q)^+\]
    on $S$. The subgroup $G_S(\Q)_+/Z_S(\Q)$ embeds into to both $G_S(\mb{A}_f)/Z_S(\Q)^-$ and $G^{\ad}_S(\Q)^+$ and its action on $S$ is the same through both embeddings. Thus, we get an action of
    \[\ms{A}(G_S) \coloneqq G_S(\mb{A}_f)/Z_S(Q)^-*_{G_S(\Q)_+/Z_S(\Q)} G_S^{\ad}(\Q)^+\]
    on $X$. We also define a subgroup
    \[\ms{A}(G_S)^\circ \coloneqq G_S(\Q)_+^-/Z_S(\Q)^- *_{G_S(\Q)_+/Z_S(\Q)} G_S^{\ad}(\Q)^+,\]
    where $G_S(\Q)_+^-$ is the closure of $G_S(\Q)_+$ in $G_S(\mb{A}_f)$.

    After taking the quotient of $S$ by $U_{S, p} = G_{S, p}(\Z_p)$, we get a right action of
    \[\ms{A}(G_{S, p}) = G_S(\mb{A}_f^p)/Z_{S, p}(\Z_{(p)})^- *_{G_{S, p}(\Z_{(p)})_+/Z_{S, p}(\Z_{(p)})} G_{S, p}^{\ad}(\Z_{(p)})^+\]
    on $S_{U_{S, p}}$, where $Z_{S, p}(\Z_{(p)})^-$ is the closure inside $G_S(\mb{A}_f^p)$. We also define the subgroup
    \[\ms{A}(G_{S, p})^\circ = G_{S, p}(\Z_{(p)})_+^-/Z_{S, p}(\Z_{(p)})^- *_{G_{S, p}(\Z_{(p)})_+/Z_{S, p}(\Z_{(p)})} G_{S, p}^{\ad}(\Z_{(p)})^+.\]
    
    Fix a geometrically connected component $S^+ \subset S$ as the image of the product of upper half planes $(\mc{H}^+)^\Sigma$ in the complex uniformization of $S$. Then take
    \[S^+ = \varprojlim_{U_{S, p}U_S^p} S_{U_{S, p}U_S^p}^+ = G_S^{\der}(\Q)\backslash (\mc{H}^+)^\Sigma \times G_S^{\der}(\mb{A}_f),\]
    and
    \[S_{U_{S, p}}^+ = \varprojlim_{U_S^p} S_{U_{S, p}U_{S}^p}^+ = G_{S, p}^{\der}(\Z_{(p)})^- \backslash (\mc{H}^+)^\Sigma \times G_S^{\der}(\mb{A}_f^p).\]
    
    Let $E_S$ be the reflex field of $S$ and let $E_S^p \subset \clos{E_S}$ be the maximal extension of $E$ that is unramified over all primes dividing $p$. The connected component $S^+$ is defined over $\clos{E_S}$ and  $S_{U_{S, p}}^+$ is defined over $E_S^p$. Let
    \[\ms{E}(G_S) \subset \ms{A}(G_S) \times \Gal(\clos{E_S}/E_S)\]
    be the stabilizer of $S^+$ and let
    \[\ms{E}(G_{S, p}) \subset \ms{A}(G_{S, p}) \times \Gal(E_S^p/E_S)\]
    be the stabilizer of $S_{U_{S, p}}^+$. Then, we have the following.

    \begin{prop}[{\cite[Lem. 3.3.7]{Kis10}}]
        The stabilizer $\ms{E}(G_S)$ (resp. $\ms{E}(G_{S, p})$) depends only on $G_S^{\der}$ (resp. $G_{S, p}^{\der}$) and $X^{\ad}$, and it is an extension of $\Gal(\clos{E_S}/E_S)$ (resp. $\Gal(E_S^p/E_S)$) by $\ms{A}(G_S)^\circ$ (resp. $\ms{A}(G_{S, p})^\circ$). Moreover, there are canonical isomorphisms
        \[\ms{A}(G_{S}) *_{\ms{A}(G_{S})^\circ} \ms{E}(G_{S}) \cong \ms{A}(G_{S}) \times \Gal(\clos{E_S}/E_S)\]
        and
        \[\ms{A}(G_{S, p}) *_{\ms{A}(G_{S, p})^\circ} \ms{E}(G_{S, p}) \cong \ms{A}(G_{S, p}) \times \Gal(E_S^p/E_S).\]
    \end{prop}

    There is a right action of $\ms{E}(G_S)$ on $\ms{A}(G_S) \times S^+$ given by right conjugation via the map $\ms{E}(G_S) \to \ms{A}(G_S) \times \Gal(\clos{E_S}/E_S) \to \ms{A}(G_S)$ on the first factor and right multiplication on the second factor. There is also an action of $\ms{A}(G_S)$ on $\ms{A}(G_S) \times S^+$ defined by right multiplication on the first factor and ignoring the second factor. Thus, there is an action of $\ms{A}(G_S) *_{\ms{A}(G_S)^\circ} \ms{E}(G_S) \cong \ms{A}(G_S) \times \Gal(\clos{E_S}/E_S)$ on $[\ms{A}(G_S) \times S^+]/\ms{A}(G_S)^\circ$. Similarly, we can define an action of $\ms{A}(G_{S, p}) *_{\ms{A}(G_{S, p})^\circ} \ms{E}(G_{S, p}) \cong \ms{A}(G_{S, p}) \times \Gal(E_S^p/E_S)$ on $[\ms{A}(G_{S, p}) \times S_{U_{S, p}}^+]/\ms{A}(G_{S, p})^\circ$.

    \begin{prop}[{\cite[Prop 3.3.10]{Kis10}}]
        For $S, S' \in \{X, X', X''\}$, there is an isomorphism of $E_S^p$ schemes
        \[S'_{U_{S', p}} \cong [\ms{A}(G_{S', p}) \times S_{U_{S, p}}^+]/\ms{A}(G_{S, p})^\circ\]
        that respects $\Gal(E_{S'}^p/E_{S'})$ action, where the Galois group acts on the right via the isomorphism $\ms{A}(G_{S', p}) *_{\ms{A}(G_{S', p})^\circ} \ms{E}(G_{S', p}) \cong \ms{A}(G_{S', p}) \times \Gal(E_{S'}^p/E_{S'})$.
    \end{prop}

    Here, we have that $\ms{A}(G_{S, p})^\circ$ acts on $\ms{A}(G_{S', p})$ via $\ms{A}(G_{S, p})^\circ \cong \ms{A}(G_{S', p})^\circ \to \ms{A}(G_{S', p})$. If there is an integral model $\mc{S}_{U_{S, p}}$ for $S$ such that the action of $G^{\ad}$ extends to it, then this isomorphism gives us a way to transfer it to all other $S'$ by simply defining
    \[\mc{S'}_{U_{S', p}} \coloneqq [\ms{A}(G_{S', p}) \times \mc{S}_{U_{S, p}}^+]/\ms{A}(G_{S, p})^\circ\]
    as $\mc{O}_{E_{S'}^p}$-schemes, and then using Galois descent to descend an $\mc{O}_{E_{S'}}$-scheme.
    
    \begin{thm}
        Let $v \mid p$ be a prime of $E_X$ and $v'' \mid p$ be a prime of $E_{X''}$. If $U^p \subset G(\mb{A}_f^p)$ and $U''^p\subset G''(\mb{A}_f^p)$ are such that $U_pU^p \subset U(N)$ and $U''_pU''^p \subset U''(N)$ for some $N \ge 3$, then there is a projective system of integral models $\mc{X}_{U_pU^p}$  (resp. $\mc{X}''_{U''_pU''^p}$) of $X_{U_pU^p}$ (resp. $X''_{U''_pU''^p}$) over $\mc{O}_{E_X, v}$ (resp. $\mc{O}_{E_{X''}, v''}$) such that:
        \begin{enumerate}
	        \item If $p$ is not ramified in $F$ nor $B$, the scheme $\mc{X}_{U}$ (resp. $\mc{X}''_{U''}$) is smooth over $\mc{O}_{E_{X}, v}$ (resp. $\mc{O}_{E_{X''}, v''}$);
            \item The schemes $\mc{X}_U$ and $\mc{X}''_{U''}$ are normal, flat, and their non-smooth locus has codimension at least $2$;
	        \item The $p$-adic completion of $\mc{X}_{U}$ (resp. $\mc{X}''_{U''}$) has a $p$-adic uniformization by a Rapoport--Zink space.
	    \end{enumerate}
    \end{thm}
    \begin{proof}
        When $E_{X', v'}$ is unramified over $\Q_p$, then the group $G'$ has a hyperspecial local model over $\mc{O}_{E_{X'}, v'}$. By \cite[Lem. 3.4.5]{Kis10}, the extension property implies that the action of $\ms{A}(G')$ extends to the integral model $\mc{X}'_{U'_p}$. Let $\mc{X'}_{U'}^+$ be the closure of ${X'}_{U'}^+$ in $\mc{X}'_{U'}$ and let
        \[\mc{X'}_{U'_p}^+ \coloneqq \varprojlim_{U'^p} \mc{X'}_{U'_p{U'}^p}^+.\]
        We then define
        \[\mc{X}_{U_p} \coloneqq [\ms{A}(G_p) \times \mc{X'}_{U'_{p}}^+]/\ms{A}(G'_p)^\circ\]
        and
        \[\mc{X''}_{U''_p} \coloneqq [\ms{A}(G''_p) \times \mc{X'}_{U'_{p}}^+]/\ms{A}(G'_p)^\circ.\]
        The action of $\ms{E}(G'_p)$ on ${X'}_{U'_p}^+$ extends to $\mc{X'}_{U'_p}^+$ and hence we get an action of
        \[\ms{A}(G_p) *_{\ms{A}(G'_p)^\circ} \ms{E}(G'_p) \cong \ms{A}(G_p) \times \Gal(E^p_X/E_X)\]
        on $\mc{X}_{U_p}$ (resp. $\ms{A}(G''_p) \times \Gal(E^p_{X''}/E_{X''})$ on $\mc{X}''_{U_p''}$). Thus, we can use the Galois action to descend this scheme to an integral model for $X_{U_p}$ (resp. $X''_{U''_p}$) defined over $\mc{O}_{E_X, v}$ (resp. $\mc{O}_{E_{X''}, v''}$). When $p$ is not ramified in $F$ nor $B$, then the integral model $\mc{X}'_{U'_p}$ corresponds to unramified PEL datum and is smooth.
        
        If $p$ is ramified in $E_{X'}$, then $G'_p$ is no longer hyperspecial and there is no extension property. However, the action of $\ms{E}(G'_p)$ still extends to $\mc{X'}_{U'_p}^+$ by \cite[Cor. 4.6.15]{Kis18}.
        
        The local rings of $\mc{X}$ and $\mc{X}''$ are \'etale locally isomorphic to the local rings of $\mc{X}'$ and $\mb{M}^{\mathrm{loc}}$. Thus, the second and third statement follow from the corresponding statements holding for $\mc{X}'_{U'_p}$ and $\mb{M}^{\mathrm{loc}}$.
    \end{proof}

    By gluing these models together, we have an integral model $\mc{X}_{U}$ over $\Spec \mc{O}_{E_X}$ for whenever $U = \prod_p U_p \subset U(N)$ for $N \ge 3$ and $U_p = U_p(0)$ is maximal whenever $p$ is ramified in $E_{X}$. We now extend the integral model for $\mc{X}_U$ when $U = \prod_p U_p$ is maximal at all primes. Take a prime $p$ that is not ramified in $E_X$ such that $U(p) = U_p(1)U^p$ is maximal at all primes away from $p$. Then define
    \[\mc{X}_{U} \coloneqq \mc{X}_{U(p)}/(U/U(p))\]
    as the quotient stack. Since the $\mc{X}_{U_pU^p}$ form a projective system, the definition of $\mc{X}_U$ does not depend on the choice of prime $p$. We note that $\mc{X}_U$ is a Deligne-Mumford stack and so it has an \'etale cover by schemes. So, all of our proofs can and will be reduced to the case of schemes via \'etale descent.
    
    \section{\texorpdfstring{$p$}{p}-Divisible Groups}\label{sec:p divisible group}
    
    When $U' = U'_pU'^p \subset U'(N)$, the functor $\mc{F}'_{U'}$ is represented by $\mc{X}'_{U'}$ and so we get a universal abelian scheme $\mc{A}'_{U'} \to \mc{X}'_{U'}$ lying over it. We use the ideas of \cite{Kis10, Kis18} to transport the $p$-divisible group $\mc{H}'_{U'} \coloneqq \mc{A}'_{U'}[p^\infty]$ to $p$-divisible groups over $\mc{X}_U$ and $\mc{X}''_{U''}$. In order to do so, we first give a description of $\mc{H}'_{U'}$ over $X'^+_{U'}$, and then describe an action of $\ms{E}(G')$ on it.
    

    Recall that $\ms{A}(G')^\circ$ depends only on the derived group $G'^{\der} = \Res_{F/\Q} B^1$, elements of norm $1 \in \Q$. The center is $Z(G'^{\der}) = \Res_{F/\Q} F^1$, elements of $F$ of norm $1$. By Shapiro's lemma, the adjoint map $G'(\Q) \to G'^{\ad}(\Q)$ is surjective and so we can write
    \[\ms{A}(G')^\circ = B^1/F^1.\]
    We can also determine $\ms{A}(G_S)$ for $S = X, X', X''$. We have that $G^{\ad}(\Z_{(p)})^+ = \mc{O}_{B, (p)}^{\times, +}/\mc{O}_{F, (p)}^{\times, +}$, where the $+$ superscript denotes the elements with norm that is totally positive in $F$. We also have that $G''^{\ad}(\Z_{(p)})^+ = (\mc{O}_{B, (p)}^{\times, +} \times_{\mc{O}_{F, (p)}^{\times, +}} \mc{O}_{E, (p)}^{\times, +})/\mc{O}_{E, (p)}^{\times, +} \cong \mc{O}_{B, (p)}^{\times, +}/\mc{O}_{F, (p)}^{\times, +}$ because the norm of all elements of $E^\times$ are totally positive in $F$. Thus, it follows that
    \[\ms{A}(G) = G(\mb{A}_f)/Z(\Q)^- *_{G(\Q)_+/Z(\Q)}G^{\ad}(\Q)^+ = (B \otimes_\Q \mb{A}_f)^\times/\mc{O}_{F}^{\times, -},\]
    and
    \[\ms{A}(G_p) = G(\mb{A}_f^p)/Z(\Z_{(p)})^- *_{G(\Z_{(p)})_+/Z(\Z_{(p)})}G^{\ad}(\Z_{(p)})^+ = (B^p)^\times/\mc{O}_{F, (p)}^{\times, -},\]
    where $B^p \coloneqq B \otimes_\Q \mb{A}_f^p$. Similarly, it follows that $\ms{A}(G'') = G''(\mb{A}_F)/\mc{O}_{E}^{\times, -}$ and $\ms{A}(G''_p) = G''(\mb{A}_f^p)/\mc{O}_{E, (p)}^{\times, -}$. Shapiro's lemma does not apply to $G'$, but we can write
    \[\ms{A}(G') = G'(\mb{A}_f)B^{\times, +}/E^{\times, -}\]
    and
    \[\ms{A}(G'_p) = G'(\mb{A}_f^p)G''(\Z_{(p)})_+/\mc{O}_{E, (p)}^{\times, -}.\]
    Moreover, we have
    \[\ms{A}(G'_p)^\circ \cong \ms{A}(G^{\der}_p)^\circ  \cong \mc{O}_{B, (p)}^1/\mc{O}_{F, (p)}^1\]

    Over $X'^+$, the $p$-divisible group $\mc{H}'$ can be described as
    \[H'|_{X'^+} = B_p/\mc{O}_{B, p} \times X'^+ = B_p/\mc{O}_{B, p} \times B^1 \backslash [(\mc{H}^+)^\Sigma \times (B \otimes_\Q \mb{A}_f)^1].\]
    To descend down to $X'^+_{U'_p}$, we quotient by the action of $G'^{\der}_p(\Z_p) = \mc{O}_{B, p}^1$ to get
    \[H'^+_{U'_p} \coloneqq H'|_{X'^+_{U'_p}} = [B_p/\mc{O}_{B, p} \times X'^+]/\mc{O}_{B, p}^1,\]
    where $U'_p = G'^{\der}_p(\Z_p) = \mc{O}_{B, p}^1$ acts on $B_p/\mc{O}_{B, p}$ by right multiplication.
    
    We can now describe $H'$ over all of $X'$. Under the isomorphism of $\clos{E_{X'}}$-schemes
    \[X' \cong [\ms{A}(G') \times X'^+]/\ms{A}(G')^\circ = [\ms{A}(G') \times X'^+]/B^1,\]
    the $p$-divisible group can be written as
    \[H'|_{X'} \cong B_p/\mc{O}_{B, p} \times [\ms{A}(G') \times X'^+]/B^1.\]
    After dividing by $G'_p(\Z_p)$, we get
    \[H'_{U'_p} |_{X'_{U'_p}} \cong [\ms{A}(G'_p) \times [B_p/\mc{O}_{B, p} \times X'^+]/\mc{O}_{B, p}^1]/\mc{O}_{B, (p)}^1 \cong [\ms{A}(G'_p) \times H'^+_{U'_p}]/\mc{O}^1_{B, (p)},\]
    where $\ms{A}(G'_p)^\circ = \mc{O}_{B, (p)}^1/\mc{O}_{F, (p)}^1 \subset (B \otimes_\Q \mb{A}_f^p)^\times/(F \otimes_\Q \mb{A}_f^p)^\times$ acts trivially on $B_p/\mc{O}_{B, p}$. In this way, we can also define $p$-divisible groups $H_{U_p}, H''_{U''_p}$ over $X_{U_p}, X''_{U''_p}$ respectively as
    \[H_{U_p}|_{X_{U_p}} \coloneqq \paren*{\ms{A}(G_p) \times [B_p/\mc{O}_{B, p} \times X'^+]/\mc{O}_{B, p}^1}/\ms{A}(G'_p)^\circ\]
    \[\cong \paren*{B^{p, \times}/\mc{O}_{F, (p)}^\times \times H'^+_{U'_p}}/\mc{O}_{B, (p)}^1,\]
    and
    \[H''_{U''_p}|_{X''_{U''_p}} \coloneqq \paren*{\ms{A}(G''_p) \times [B_p/\mc{O}_{B, p} \times X'^+]/\mc{O}_{B, p}^1}/\ms{A}(G'_p)^\circ\]
    \[\cong \paren*{G''(\mb{A}_f^p)/\mc{O}_{E, (p)}^\times \times H'^+_{U'_p}}/\mc{O}_{B, (p)}^1,\]

    \begin{thm}
        Whenever $U \subset U(N)$ (resp. $U'' \subset U''(N)$) for $N \ge 3$, there exists a $p$-divisible group $\mc{H}_{U}$ over $\mc{X}_{U}$ with an $\mc{O}_{E, p}$-action (resp. $\mc{H}''_{U''}$ over $\mc{X}''_{U''}$) such that the formal completion $\widehat{\mc{X}_U}$ (resp. $\widehat{\mc{X}''_{U''}}$) along its special fiber over $\clos{k_{E_X, v}}$ (resp. $\clos{k_{E_{X''}, v''}}$) is the universal deformation space of $\mc{H}_{\clos{k_{E_X, v}}}$ (resp. $\mc{H}_{\clos{k_{E_{X''}, v''}}}$).

    \end{thm}
    \begin{proof}
        We have already translated the $p$-divisible group $\mc{H}'_{U'_p}$ to $p$-divisible groups over the generic fiber $X_{U_p}$ and $X''_{U''_p}$. Over $\mc{X}'_{U'_p}$, we have an integral model for $\mc{H}'_{U'_p}$ by taking the $p^\infty$-torsion of the universal abelain scheme $\mc{A}'_{U'_p} \to \mc{X}'_{U'_p}$. We can restrict it to the connected component to get
        \[\mc{H}'^+_{U'_p} \coloneqq \mc{H}'_{U'_p}|_{\mc{X}'^+_{U'_p}}.\]
        Set
        \[\mc{H}_{U_p} \coloneqq \paren*{\ms{A}(G_p) \times \mc{H}'^+_{U'_p}}/\ms{A}(G'_p)^\circ,\]
        and
        \[\mc{H''}_{U''_p} \coloneqq \paren*{\ms{A}(G''_p) \times \mc{H}'^+_{U'_p}}/\ms{A}(G'_p)^\circ.\]
        The action of $\ms{E}(G^{\der}_p)$ on $\mc{X}'^+_{U'_p}$ extends to an action on $\mc{H}'^+_{U'_p}$ by acting trivially on $B_p/\mc{O}_{B, p}$ and thus there is an action of $\ms{A}(G_p) \times \Gal(E_X^p/E_X)$ on $\mc{H}_{U_p}$ that is compatible with the structure morphism $\mc{H}_{U_p} \to \mc{X}_{U_p}$. Hence, we can descend $\mc{H}_{U_p}$ and $\mc{H}''_{U''_p}$ down to $p$-divisible groups defined over $\mc{O}_{E_X, v}$ and $\mc{O}_{E_{X''}, v''}$ respectively, whose generic fibers can be identified with $H_{U_p}$ and $H''_{U''_p}$ in a way respecting the structure morphisms down to $X_{U_p}$ and $X''_{U''_p}$. For finite level, we can simply take $\mc{H}_{U_pU^p} = \mc{H}_{U_p}/U^p$ where $U^p$ acts below on $\mc{X}_{U_p}$ and acts trivially on the fibers of $\mc{H}_{U_p} \to \mc{X}_{U_p}$.
        
        The statements about universal deformation spaces and $\mc{O}_{F, (p)}$-action follow from the corresponding statements for $\mc{H}'^+_{U'}$ over $\mc{X}'^+_{U'}$. These follow from $\mc{X}'_{U'}$ representing the functor of isomorphism classes of abelian schemes whose $p^\infty$-torsion is $\mc{H}'_{U'}$.
    \end{proof}
    
    \section{Kodaira-Spencer Map}\label{sec:Hodge bundle}
    
    In order to calculate the height of a partial CM point, we will take the height of a special point of $\mc{X}_U$ with respect to metrized Hodge bundle $\widehat{\mc{L}_U}$ on $\mc{X}_U$, which we will introduce. Afterwards, we will relate this Hodge bundle with the Lie groups of the $p$-divisible group $\mc{H}_U$ over $\mc{X}_U$.
    
    Following \cite{Yua13}, we define the system of Hodge bundles $\{\mc{L}_U\}_U$ as the canonical bundle
    \[\mc{L}_U \coloneqq \omega_{\mc{X}_U/\mc{O}_{E_X}}.\]
    Since $\mc{X}_U$ is normal, the singularities of $\mc{X}_U$ have codimension at least $2$ and so the sheaf $\mc{L}_U$ is indeed a line bundle. The benefit of using the canonical bundle is that the system $\{\mc{L}_U\}_U$ is compatible with pullbacks along the canonical maps $\mc{X}_{U_1} \to \mc{X}_{U_2}$ with $U_1 \subset U_2$. We call $\mc{L}_U$ the \emph{Hodge bundle} of $\mc{X}_U$. At the infinite places, the metric is given by the Petersson metric
	\[\norm*{\bigwedge_{\sigma \in \Sigma} dz_\sigma} = \prod_{\sigma \in \Sigma} \mathrm{Im}(2z_\sigma).\]
	
	We now relate this line bundle with our $p$-divisible group via a Kodaira-Spencer type map. Suppose $K/E_X$ is a finite extension that contains the normal closure of $E$ and let $S = \Spec \mc{O}_{K} \otimes_\Z \Z_p$. Base change $\mc{X}_U$ and $\mc{H}_U$ to $S$ and set $\Omega(\mc{H}_S) \coloneqq \Lie(\mc{H}_S)^\vee$ and $\Omega(\mc{H}^t_S) \coloneqq \Lie(\mc{H}^t_S)^\vee$. Let $\mb{D}(\mc{H}_S)$ and $\mb{D}(\mc{H}_S^t)$ be the covariant Dieudonn\'e crystals attached to the $p$-divisible groups. Then \cite[Chap. IV]{Mes72} gives us a short exact sequence
	    \[0 \to \Lie(\mc{H}_S^t)^\vee \to \mb{D}(\mc{H}_S) \to \Lie(\mc{H}_S) \to 0\]
	    of $\mc{O}_S \otimes \mc{O}_{E}$ modules. Applying the Gauss--Manin connection $\nabla$ on $\mb{D}(\mc{H}_S)$ gives the chain of maps
     \[\Omega(\mc{H}_S^t) \to \mb{D}(\mc{H}_S) \overset{\nabla}{\to} \mb{D}(\mc{H}_S) \otimes \Omega_{\mc{X}_S/S}^1 \to \Omega(\mc{H}_S)^\vee \otimes \Omega_{\mc{X}_S/S}^1,\]
     which gives a map
     \[\Omega_{\mc{X}_S/S}^{1, \vee} \to \Hom(\Omega(\mc{H}_S^t), \Omega(\mc{H}_S)^\vee).\]
    
    Both $\Omega(\mc{H}_S)^\vee$ and $\Omega(\mc{H}^t_S)^\vee$ have an action by $\mc{O}_E$ whose determinant is the product of the reflex norms of $\phib \sqcup \phib'$ and $\conj{\phib} \sqcup \phib'$. We can thus decompose the line bundles over $S$ as
	\[\Omega(\mc{H}_S)^\vee \to \bigoplus_{\tau \in \Hom(E, \C)} \Omega(\mc{H}_S)^\vee_\tau,\]
    where $\Omega(\mc{H}_S)_\tau^\vee \coloneqq \Omega(\mc{H}_S)^\vee \otimes_{\mc{O}_S \otimes \mc{O}_E, \tau} \mc{O}_S$ and $\mc{O}_E$ acts on $\mc{O}_S$ through fixing an inclusion of $\clos{E_X} \to \C$. Set $\omega(\mc{H}_S)_\tau \coloneqq \det\Omega(\mc{H}_S)_\tau$. We define $\Omega(\mc{H}_S^t)_\tau$ and $\omega(\mc{H}_S^t)_\tau$ similarly. Thus, we get a map
    \[\Omega_{\mc{X}_S/S}^{1, \vee} \to \Hom(\Omega(\mc{H}_S^t), \Omega(\mc{H}_S)^\vee) \to \bigoplus_{(\tau, \tau') \in \Hom(E, \C)^2} \Omega(\mc{H}_S^t)_\tau^\vee \otimes \Omega(\mc{H}_S)^\vee_{\tau'}\]
    The rank of $\Omega(\mc{H}_S)_\tau$ is $1$ if $\tau \in \phib$, $2$ if $\tau \in \phib'$, and $0$ if $\tau \in \conj{\phib'}$. The rank of $\Omega(\mc{H}_S^t)_\tau$ is $2 - \dim \Omega(\mc{H}_S)_\tau$. So, by projecting, we get the map of vector bundles of equal rank
    \[\Omega_{\mc{X}_S/S}^{1, \vee} \to \bigoplus_{\tau \in \phib} \omega(\mc{H}_S^t)_\tau^\vee  \otimes \omega(\mc{H}_S)_\tau^\vee.\]
    
    Taking the determinant of this map and repeating the map for $\conj{\phib}$ gives the map
    \[\omega_{\mc{X}_S/S}^{-2} \to \bigotimes_{\tau \in \phib} \mc{N}(\mc{H}_S, \tau)^\vee \otimes \mc{N}(\mc{H}_S, \conj{\tau})^\vee,\]
    where $\mc{N}(\mc{H}_S, \tau) \cong \omega(\mc{H}_S)_\tau \otimes \omega(\mc{H}_S^t)_{\conj{\tau}}$.

    Note that $E_X = E_{\phib \sqcup \conj{\phib}}$ is the reflex field of the set $\phib \sqcup \conj{\phib}$. So, we can descend this map to a map over $\Spec \mc{O}_{E_X, p}$. This is our Kodaira-Spencer map.
    
    Let $\Sigma_E \subset \Hom(E, \C)$ be the places that lie above a place of $F$ in $\Sigma$. Recall from Definition \ref{def:CM Type Determinant} the relative discriminant $\mf{d}_{\Sigma} \coloneqq \mf{d}_{\Sigma_E}$. Let $\mf{d}_{\Sigma, p} \coloneqq \mf{d}_{\Sigma} \otimes \Z_p$ be the $p$-part of this ideal. View $\mf{d}_{\Sigma, p}$ as a divisor of $S$. Moreover, let $\mf{d}_{B, p}$ be the divisor corresponding to the ramification of $B$ above $p$.

    \begin{thm}\label{thm:Hodge Bundle p Divisible Group}
	    For any choice of partial CM-type $\phib$ lying above $\Sigma$, we have
	    \[\omega_{\mc{X}_S/S}^2(\mf{d}_{\Sigma, p}\mf{d}_{B, p})^{-1} \cong \bigotimes_{\tau \in \phib} \mc{N}(\mc{H}_S, \tau) \otimes \mc{N}(\mc{H}_S, \conj{\tau}).\]
	\end{thm}
	\begin{proof}
	    We have the short exact sequence
	    \[0 \to \Lie(\mc{H}_S^t)^\vee \to \mb{D}(\mc{H}_S) \to \Lie(\mc{H}_S) \to 0\]
	    of $\mc{O}_S \otimes \mc{O}_{E}$ modules. Moreover, we have a pairing
     \[\mb{D}(\mc{H}_S) \times \mb{D}(\mc{H}_S^t) \to S\]
     that is well defined up to an element of $S^\times$. Fix a choice. The formal completion of $\mc{X}_S/S$ along its special fiber over $\clos{k(S)}$ is the universal deformation space of $\mc{H}_S$. Thus, \cite{Mes72} gives us that the tangent bundle $\Omega^{1, \vee}_{\mc{X}_S/S}$ corresponds to choosing a lift of $\Lie(\mc{H}_S^t)^\vee$ and $\Lie(\mc{H}_S)$ in $\mb{D}(\mc{H}_S)_{S'}$, where $S' = \Spec \mc{O}_{S}[\epsilon]/(\epsilon^2)$, that respects the pairing from $\psi_S$ and $\mc{O}_S \otimes \mc{O}_E$ action.

        For each $\tau \in \Hom(E, \C)$, we can take the $\tau$-component of the short exact sequence to get
        \[0 \to \Omega(\mc{H}_S^t)_{\tau} \to \mb{D}(\mc{H}_S)_{\tau} \to \Omega(\mc{H}_S)_{\tau}^\vee \to 0.\]
        For $\tau$ lying above $\Sigma^c$, either $\Omega(\mc{H}_S^t)_\tau$ or $\Omega(\mc{H}_S)_\tau$ is $0$ meaning that there is only one choice for a lift of $\Omega(\mc{H}_S^t)_\tau$ and $\Omega(\mc{H}_S)_\tau$. For $\tau$ lying above $\Sigma$, both are of rank $1$. The pairing $\psi_{S'}$ decomposes into an orthogonal sum of pairings
        \[\mb{D}(\mc{H}_S)_{S', \tau} \times \mb{D}(\mc{H}_S^t)_{S', \conj{\tau}} \to S'.\]
        Thus, choosing a lift of $\Omega(\mc{H}_S)_\tau$ determines the choice for $\Omega(\mc{H}_S)_{\conj{\tau}}$ under the canonical isomorphism $\mb{D}(\mc{H}_S^t) \cong \mb{D}(\mc{H}_S)^\vee$. The Hodge filtration gives us that the choice of lift of $\Omega(\mc{H}_S^t)_\tau$ is a torsor of $\Hom(\Omega(\mc{H}_S^t)_\tau, \Omega(\mc{H}_S)_\tau^\vee)$ giving us that the map
        \[\Omega_{\mc{X}_S/S}^{1, \vee} \to \Hom(\Omega(\mc{H}_S^t), \Omega(\mc{H}_S)^\vee) \to \bigoplus_{\tau \in \phib} \Hom(\Omega(\mc{H}_S^t)_\tau, \Omega(\mc{H}_S)_\tau^\vee)\]
        has finite cokernel. Taking determinants and repeating the process with $\conj{\phib}$ instead of $\phib$, we get
        \[\omega_{\mc{X}_S/S}^{-2} \to \bigotimes_{\tau \in \phib} \omega(\mc{H}_S^t)^\vee_\tau \otimes \omega(\mc{H}_S)_\tau^\vee \otimes \omega(\mc{H}_S^t)^\vee_{\conj{\tau}} \otimes \omega(\mc{H}_S)_{\conj{\tau}}^\vee.\]
        The cokernel of this map classifies the failure of when choosing a lift of $\Omega(\mc{H}_S)_\tau$ for each $\tau$ does not arise from choosing a lift of $\Omega(\mc{H}_S)$, and when choosing a lift of $\Omega(\mc{H}_S)_\tau$ does not result in a lift of $\Omega(\mc{H}_S)_{\conj{\tau}}$ due to when the pairing $\psi_{S'}$ is not perfect.

        Let $\pi \in \mc{O}_S$ be a generator for $\mc{O}_S$ over $\mc{O}_{E_X, p}$. For a subset $\Psi \subset \Hom(E, \C)$, let
        \[f_\Psi(t) = \prod_{\tau \in \Psi} (t - \tau(\pi)).\]
        We see that $f_{\phib \cup \conj{\phib}}(t) \in \mc{O}_{E_X, p}[t]$ because it is invariant under any automorphism that fixes the underlying places of $F$ under $\phib$. Then, we see that the image of $\mc{O}_{E_X, p} \otimes_\Z \mc{O}_E$ in $\widetilde{E_{\phib \sqcup \conj{\phib}, p}}$ is simply $\mc{O}_{E_X, p}[t]/f_{\phib \sqcup \conj{\phib}}(t)$, and so $\mf{d}_{\phib\sqcup \conj{\phib}, p} \subset \mc{O}_{E_X, p}$ is the ideal generated by the discriminant of $f_{\phib \sqcup \conj{\phib}}$,
	    
        By \cite[Cor. 2.5]{Yua18}, we have that
        \[\Omega(\mc{H}_S) \cong \frac{\mc{O}_S[t]}{f_{\phib \sqcup \phib'}(t)f_{\conj{\phib}\sqcup \phib'}(t)}, \qquad \Omega(\mc{H}_S^t) \cong \frac{\mc{O}_S[t]}{f_{\conj{\phib} \sqcup \conj{\phib'}}(t)f_{\phib\sqcup \conj{\phib'}}(t)}.\]
        Each element of $\omega_{\mc{X}_S/S}^\vee$ corresponds to an element of $\Omega(\mc{H}_S^t)^\vee \otimes \Omega(\mc{H}_S)^\vee$. So, the image in $\bigotimes_{\tau \in \phib} \omega(\mc{H}_S^t)^\vee_\tau \otimes \omega(\mc{H}_S)_\tau^\vee$ is the determinant of the image of $\Omega(\mc{H}_S^t) \otimes \Omega(\mc{H}_S)$ in
        \[\prod_{\tau \in \phib \cup \conj{\phib}} \Omega(\mc{H}_S^t)_\tau \otimes \Omega(\mc{H}_S)_\tau\cong \prod_{\tau \in \phib \sqcup \conj{\phib}} \mc{O}_S[t]/(t - \tau(\pi)) \otimes \mc{O}_S[t]/(t - \tau(\pi))\]
        under the map $t\mapsto (t, t, \dots, t)$. We deal with $\Omega(\mc{H}_S)$ and $\Omega(\mc{H}_S^t)$ separately. A basis for $\Omega(\mc{H}_S)$ is given by $1, t, \dots, t^{2\abs{\phib} - 1}$, so the lattice formed by the image of $\Omega(\mc{H}_S)$ in $\prod_{\tau \in \phib \sqcup \conj{\phib}} \Omega(\mc{H}_S)_\tau \cong \prod_{\tau \in \phib \sqcup \conj{\phib}} \mc{O}_{S, \tau}$ is generated by $\{(\tau(\pi)^i)_{\tau \in \phib \cup \conj{\phib}}\}_{0 \le i < 2\abs{\phib}}$. To calculate the index of this lattice relative to the maximal lattice, we take the determinant of a $2\abs{\phib} \times 2\abs{\phib}$ matrix whose $ij$-th element is $\tau_i(\pi)^j$. The ideal generated by the determinant of this Vandermonde matrix is
        \[\paren*{\prod_{1 \le i < j \le 2\abs{\phib}} \abs{\tau_j(\pi) - \tau_i(\pi)}} = \mf{d}_{\Sigma, p}^{1/2}.\]
        Doing the same for $\Omega(\mc{H}_S^t)$ nets an additional factor of $\mf{d}_{\Sigma, p}^{1/2}$.

        Finally, when $B$ is ramified, the pairing $\mb{D}(\mc{H}_S)_{S'} \times \mb{D}(\mc{H}_S^t)_{S'} \to S'$ is not perfect but rather $\Lambda^\vee = \frac{1}{\omega_\mf{q}}\Lambda$, meaning our choice in $\Omega(\mc{H}_S)$ must lie in $\omega_\mf{q} \Omega(\mc{H}_S)$, giving an additional factor of $(\omega_\mf{q}) = \mf{d}_{B, p}^{1/2}$ since $\mf{p}$ was specified to be ramified wherever $B$ was. Doing the same for $\mc{H}_S^t$ nets us another factor of $\mf{d}_{B, p}^{1/2}$.
	\end{proof}

	\section{Special Points}\label{sec:Special Point}
	
	To calculate heights of special points of $\mc{X}_U$, we relate the height to heights on $\mc{X}'_{U'}$, which represent Faltings heights, through $\mc{X}''_{U''}$. Let $(E, \phib)$ be a partial CM-type with $F \subset E$ the totally real subfield of index $2$ and let $\Sigma = \phib|_F \subset \Hom(F, \R)$. Let $B$ be a quaternion algebra over $F$ such that $B$ is ramified at infinity at $\Sigma \subset \Hom(F, \R)$ and whose finite ramification set is a subset of the primes for which $E$ is ramified. Then we can embed $E \to B$ because $E_\mf{p}$ embeds into $B_\mf{p}$ at every place $\mf{p}$ of $F$. Let $\{\mc{X}_U\}_U$ be the tower of Shimura varieties associated to this particular quaternion algebra $B$. The embedding $E \to B$ gives us an embedding of $T_E \coloneqq \Res_{E/\Q} \mb{G}_m \to G$ and hence we get a set of CM points of $X_U$ which are parametrized, under the complex uniformization, by points $(z, t) \in (\mc{H}^{\pm})^\Sigma \times G(\mb{A}_f)$ where $z$ is determined by the cocharacter $h_\sigma\colon \C \cong E_\tau \to B_\sigma$ for each $\tau \in \phib$ and $\sigma \colon F \to \R$ lying below it, and $t \in T_E(\mb{A}_f) \subset G(\mb{A}_f)$. Fix one of these CM points $P \in \mc{X}_U(\clos{\Q})$.
	
	Pick a complementary partial CM-type $\phib'$ to $\phib$. We can construct the tower $X'_{U'}$ which represents the functor $\mc{F}'_{U'}$ as before. For any choice of element $t' \in [(T_E \times T_E) \cap G'](\mb{A}_f)$, the cocharacter formed from $z \in (\mc{H}^\pm)^\Sigma$ and $h_E$ is a point $P' = [(z, h_E), t'] \in X'_{U'}(\clos{\Q})$, which represents an abelian variety $A'$ with multiplication by $\mc{O}_{E}$. From the determinant condition of $\mc{F}'_{U'}$, we have that $A'$ is isogenous to a product of abelian varieties $A_1 \times A_2$, one with CM by $E$ of type $\phib \sqcup \phib'$ and the other of CM-type $\conj{\phib}\sqcup \phib'$.
	
	We now compare the points $P$ and $P'$ by embedding both $X_U$ and $X'_{U'}$ into $X''_{U''}$. Recall that $G'' = \Res_{F/\Q} [(B^\times \times E^\times)/F^\times]$ where $F^\times \subset B^\times \times E^\times$ by $a \mapsto (a, a^{-1})$. This gives rise to the Shimura variety $X''_{U''}$. The embedding $G' \to G''$ gives an embedding $X'_{U'} \to X''_{U''}$. To relate $X_U$ and $X''_{U''}$, we take the group $\Res_{F/\Q} (B^\times \times E^\times)$ and the quotient map $\Res_{F/\Q} (B^\times \times E^\times) \to G''$. The group gives rise to a Shimura variety $X_U \times Y_J$, where $Y_J$ is the zero-dimensional Shimura variety associated with the datum of the torus $T_E$ and morphism $h_E$ as in the definition of $G''$, and $J \subset T_E(\mb{A}_f)$ an open compact subgroup. This quotient map of Shimura datum gives rise to a surjective morphism
	\[X_U \times Y_J \to X''_{U''}\]
	 of Shimura varieties, where $U'' = U \cdot J \subset G''(\mb{A}_f)$. Thus, we have the following morphisms of algebraic groups
  \[G \gets G \times T_E \to G'' \gets G'\]
  which gives rise to the chain of morphisms of Shimura varieties
  \[X_U \gets X_U \times Y_J \to X''_{U''} \gets X'_{U'}.\]
  However, given a point $y \in Y_J$, we are able to construct a section $X_U \to X_U \times Y_J \to X''_{U''}$.

    \begin{prop}
        Suppose $P \in X_U$ is a CM point corresponding to the embedding $E \hookrightarrow B$. We can choose $y \in Y_J$ and $P' \in X'_{U'}$ such that $P \in X_U$ and $P' \in X'_{U'}$ have the same image $P''$ in $X''_{U''}$.
    \end{prop}
    \begin{proof}
        For a given $P \in X_U$, we can choose a representative $[z, t] \in (\mc{H}^\pm)^\Sigma \times G(\mb{A}_f)$ under the complex uniformization. If we let $t' = (t, t^{-1})$, then $t' \in (T_E \times T_E \cap G')(\mb{A}_f)$ because $\nu(t') = 1 \in \mb{G}_m$. Letting $y \in Y_J$ be the point corresponding to the choice of $t^{-1} \in T_E(\mb{A}_f)$ makes it so $(P, y) \in X_U \times Y_J$ and $P' = [z \times h_E, t'] \in X'_{U'}$ have the same image in $X''_{U''}$.
    \end{proof}
	
	All of the geometric points of $Y_J$ are defined over $E_{X'}$, so the integral model $\mc{X}_U$ for $X_U$ gives rise to an integral model $\mc{X}_U \times \mc{Y}_J$ for $X_U \times Y_J$. We have a $p$-divisible group $\mc{H}''_{U''}$ defined over $\mc{X}''_{U''}$. We also define a $p$-divisible group $I$ over $Y_J$ by setting
	\[I_J \coloneqq (E_p/\mc{O}_{E, p} \times Y)/J.\]
	
%
	
	Let $K/E_{X'}$ be a finite extension. Suppose that we have points $x \in X_U(K)$ and $y \in Y_J(K)$. These give rise to a point $x'' \in X''_{U''}(K)$. By \cite[Prop 5.2]{Yua18}, we can extend $I_J$ to a $p$-divisible group $\mc{I}_y$ over the closure of the point $y$ in $\mc{Y}_J$.
	
	\begin{prop}[{\cite[Prop 5.3]{Yua18}}]\label{prop:compare p divisible groups}
	        There are canonical isomorphisms
	        \[\Lie(\mc{H}''_{x''}) \cong \Lie(\mc{H}_x) \otimes_{\mc{O}_{E, p} \otimes \mc{O}_{K}} \Lie(\mc{I}^t_{y})^\vee,\quad \Lie(\mc{H}''^t_{x''}) \cong \Lie(\mc{H}^t_x) \otimes_{\mc{O}_{E, p} \otimes \mc{O}_{K}} \Lie(\mc{I}^t_{y}).\]
	\end{prop}

        Define $\mc{N}''(\mc{H}''_{x''}, \tau) \coloneqq \omega(\mc{H}''_{x''})_\tau \otimes \omega(\mc{H}''_{x''})_{\conj{\tau}}$. Then the previous proposition immediately gets us that
        \[\mc{N}''(\mc{H}''_{x''}, \tau) \cong \mc{N}(\mc{H}_x, \tau).\]

        Since the point $P$ and $P'$ coincide in $\mc{X}''_{U''}$, the $p$-divisible group $\mc{H}_P$ coincides with the $p$-infinity torsion of the abelian scheme $\mc{A}'_{P'}$ over $P' \in \mc{X}'_{U'}$. Thus, we can use the same norm from the Hermitian pairing
        \[\norm{\cdot} \colon W(A'_{P'}, \tau) \otimes W(A'^t_{P'}, \conj{\tau}) \to \C\]
        to give
        \[\widehat{\mc{N}(\mc{H}_P, \tau)} \coloneqq \paren*{\omega(\mc{H}_P)_\tau \otimes \omega(\mc{H}_P^t)_{\conj{\tau}}, \norm{\cdot}}\]
        into a metrized line bundle.

        \begin{thm}
            The Kodaira--Spencer isomorphism in Theorem \ref{thm:Hodge Bundle p Divisible Group} respects the norms at infinity at $P$ and hence extends to an isomorphism of metrized line bundles
            \[\widehat{\mc{L}_P}^2(\mf{d}_{\Sigma, p}\mf{d}_{B, p})^{-1} \cong \bigotimes_{\tau \in \phib} \widehat{\mc{N}(\mc{H}_P, \tau)} \otimes \widehat{\mc{N}(\mc{H}_P, \conj{\tau}}).\]
        \end{thm}
        \begin{proof}
            At the places at infinity, the Dieudonn\'e module $\mb{D}(\mc{H}_P)$ is naturally isomorphic to the first de Rham homology of $A_P \coloneqq A'_{P'}$. Thus, the Kodaira--Spencer morphism comes from the Hodge filtration
            \[0 \to \Omega(A^t_P) \to H_1^{\dR}(A_P) \to \Omega(A_P) \to 0.\]
            For each $\tau\colon E \to \C$, we can look at the $\tau$-component of the filtration
            \[0 \to \Omega(A^t_P)_\tau \to H_1^{\dR}(A_P)_\tau \to \Omega(A_P)_\tau \to 0\]
            For $\tau\colon E \to \C$ lying above $\Sigma^c$, there is no contribution from either line bundle so we can restrict ourselves to considering the $\tau$ lying above $\Sigma$, and specifically for $\tau \in \phib$. For these $\tau$, we have 
            \[\Omega(A^t_P)_\tau \to H_1^{\dR}(A_P)_\tau \overset{\nabla}{\to} H_1^{\dR}(A_P)_\tau \otimes \Omega^1_{X_U, P} \to \Omega(A_P)^\vee \otimes \Omega^1_{X_U, P}.\]
            Explicitly, we have that $H_1^{\dR}(A_P)_\tau \cong V_\tau \cong B \otimes_{E, \tau} \C$. We can choose an isomorphism $B \otimes_{E, \tau} \C \cong \C \oplus \C$ so that $h(i)_\tau \subset M_2(\R)$ acts on $V_\tau$ via right transpose action $h(i) \cdot (z_1, z_2) = (z_1, z_2) \conj{h(i)}$. Then in terms of the complex uniformization, an element $z = x + iy \in \mb{H}_\tau^\pm$ in the complex half planes corresponds to a conjugate of $h(i)$ and $\Omega(A^t_P)_\tau \cong V^{0, -1}_\tau$ is the subset of $\C^2$ for which $h(i)$ acts as $-i$. Computation shows that $\Omega(A^t_P)_\tau \cong \C(z, 1) \subset \C^2$. Moreover, we have that $\Lie(A/P) = V^{-1, 0} \cong \C(\conj{z}, 1) \subset \C^2$. Thus, explicitly, the map above gives
            \[
            \begin{tikzcd}[row sep = small, column sep=small]
                \Omega(A^t_P)_\tau \arrow[r] & H_1^{\dR}(A_P)_\tau \arrow[r, "\nabla"] & H_1^{\dR}(A_P)_\tau \otimes \Omega^1_{X_U, P} \arrow[r] & \Omega(A_P)^\vee \otimes \Omega^1_{X_U, P} \\
                (z, 1) \arrow[r, maps to] & (z, 1) \arrow[r, maps to] & (1, 0) \otimes dz = \frac{(z, 1) - (\conj{z}, 1)}{2iy} \otimes dz \arrow[r, maps to] & \frac{-(\conj{z}, 1)}{2iy} \otimes dz.
            \end{tikzcd}
            \]

            Thus at infinity, the isomorphism $\omega_{X_U, P} \to \bigotimes_{\tau \in \phib} \omega(A^t_P)_\tau \otimes \omega(A_P)_\tau$ gives
            \[\bigwedge_{\tau \in \phib} dz \mapsto \bigotimes_{\tau \in \phib} 2iy_\tau\frac{(z_\tau, 1)}{(\conj{z_\tau}, 1)},\]
            and taking norms gives $\prod_{\tau \in \phib} 2y_\tau$ on both sides.
        \end{proof}
 
	\begin{thm}\label{thm:heightP}
	    Let $d_B$ be a positive generator of $N_{F/\Q} \mf{d}_{B}$ and let $d_\Sigma = d_{\phib \sqcup \conj{\phib}}$. We have that
	    \[h_{\widehat{L_U}}(P_U) = \sum_{\tau \in \phib} \paren*{h(\phib \sqcup \phib', \tau) + h(\conj{\phib} \sqcup \phib', \conj{\tau})} + \frac{1}{2g} \log d_Bd_{\Sigma}.\]
	\end{thm}
	\begin{proof}
        By Theorem \ref{thm:Hodge Bundle p Divisible Group}, we get that
	    \[2h_{\widehat{\mc{L}_U}}(P_U) = \sum_{\tau \in \phib} h_{\widehat{\mc{N}(\tau)}}(P_U) + h_{\widehat{\mc{N}(\conj{\tau})}}(P_U) + \frac{1}{2g}\log d_Bd_\Sigma,\]
        with the extra factor of $g$ coming from the fact that we defined the height over $\Q$, which is $[F:\Q]$ times larger than the usual height defined over $F$. By the previous proposition, we have that
        \[h_{\widehat{\mc{N}(\tau)}}(P_U) = h_{\widehat{\mc{N}''(\tau)}}(P''_{U''}).\]
        Then by our choice of $y \in Y_J$ and $P' \in X'_{U'}$, the point $P''_{U''}$ is the image of $P' \in X'_{U'}$ which represents an abelian variety that is isogenous to a product of CM abelian varieties, one of CM-type $\phib \sqcup \phib'$ and the other of CM-type $\conj{\phib} \sqcup \phib'$. Thus, we get that
        \[h_{\widehat{\mc{N}''(\tau)}}(P''_{U''}) = h_{\widehat{\mc{N}'(\tau)}}(P'_{U'}) = h(\phib \sqcup \phib', \tau) + h(\conj{\phib}\sqcup \phib', \conj{\tau}).\]
	\end{proof}
	
	This result does not depend on the choice of complementary CM-type $\phib'$ and so summing over all such complementary CM-types nets us the following.
	
	\begin{thm}
	    Suppose that $U = \prod_v U_v$ is a maximal compact subgroup of $G(\mb{A}_f)$. Then
		\begin{align*}
		    \frac{1}{2}h_{\widehat{\mc{L}_U}}(P_U) =& \frac{1}{2^{\abs{\comp{\Sigma}}}} \sum_{\Phi \supset \phib} h(\Phi) - \frac{\abs{\comp{\Sigma}}}{g2^g} \sum_\Phi h(\Phi)\\
      &+ \frac{1}{8} \log d_{E/F, \Sigma}d_\Sigma^{-1} + \frac{1}{4}\log d_\phib d_{\conj{\phib}} + \frac{1}{4g} \log d_Bd_\Sigma + \frac{\abs{\Sigma}}{4g} \log d_F,
            \end{align*}
		where the first summation is over all full CM-types which contain $\phib$, and the second summation over all full CM-types of $E$.
	\end{thm}
    \begin{proof}
        We note $h(\Phi) = h(\conj{\Phi})$ and $h(\Phi, \tau) = h(\conj{\Phi}, \conj{\tau})$. So we can write
        \begin{equation}\tag{*}\label{eq:*}
            \begin{split}
            h_{\widehat{\mc{L}_U}}(P_U) - \frac{1}{2g}\log d_Bd_\Sigma= & \sum_{\tau \in \phib \sqcup \phib'} h(\phib \sqcup \phib', \tau) + \sum_{\tau \in \phib \sqcup \conj{\phib'}} h(\phib \sqcup \conj{\phib'}, \tau) \\
            &- \sum_{\tau \in \phib'} \paren*{h(\phib \sqcup \phib', \tau) + h(\phib \sqcup \conj{\phib'}, \conj{\tau})}. 
            \end{split}
        \end{equation}
        
        Let $(\Phi_1, \Phi_2)$ be a nearby pair of full CM-types meaning that $\abs{\Phi_1 \cap \Phi_2} = g-1$ and let $\tau_i = \Phi_i \backslash (\Phi_1 \cap \Phi_2)$. Then Theorem \ref{thm:nearby CM-type} tells us the quantity $h(\Phi_1, \tau_1) + h(\Phi_2, \tau_2)$ is independent of the choice of nearby pair, so we will denote it by $h_{\mathrm{nb}}$. By \cite[Cor. 2.6]{Yua18}, we have that
        \[\sum_{\Phi} h(\Phi) = g2^{g - 1}h_{\mathrm{nb}} - 2^{g - 2}\log d_F.\]
        
        Now we sum equation (\ref{eq:*}) over all complementary types $\phib'$ to get
        \[2^{\abs{\comp{\Sigma}}} h_{\widehat{\mc{L}_U}}(P_U) - \frac{2^{\abs{\comp{\Sigma}}}}{2g}\log d_Bd_\Sigma = 2\sum_{\Phi \supset \phib} \sum_{\tau \in \Phi} h(\Phi, \tau) - \abs{\comp{\Sigma}}2^{\abs{\comp{\Sigma}}} h_{\mathrm{nb}}.\]
        We now use Theorem \ref{thm:fullvscomponent} to represent the inner summation as $\sum_{\tau \in \Phi} h(\Phi, \tau)$. Doing so gives
        \begin{align*}
            \frac{1}{2}h_{\widehat{\mc{L}_U}}(P_U) =& \frac{1}{2^{\abs{\comp{\Sigma}}}} \sum_{\Phi \supset \phib} h(\Phi) - \frac{\abs{\comp{\Sigma}}}{g2^g} \sum_\Phi h(\Phi)\\
            &+ \frac{1}{2^{\abs{\comp{\Sigma}}}} \sum_{\Phi \supset \phib} \frac{1}{4[E_\Phi:\Q]} \log (d_\Phi d_{\conj{\Phi}}) + \frac{1}{4g} \log d_Bd_\Sigma - \frac{\abs{\comp{\Sigma}}}{4g} \log d_F.
        \end{align*}

        Base changing up to $E' = E^{\Gal}$, we can simplify the first sum of logarithms as
        \[\frac{1}{2^{\abs{\comp{\Sigma}}} \cdot 4[E':\Q]} \sum_{\Phi \supset \phib} \log(d_\Phi d_{\conj{\Phi}}) = \sum_{p < \infty} \sum_{\sigma\colon E' \to \clos{\Q_p}}\frac{1}{2^{\abs{\comp{\Sigma}}} \cdot 4[E':\Q]} \sum_{\Phi \supset \phib} \log \abs{d_{\Phi, p}d_{\conj{\Phi}, p}}_\sigma.\]
        For each $\tau \colon E' \to \clos{\Q_p}$, let $\pi$ be a generator of $\mc{O}_{E, p}$ over $\Z_p$. Let
        \[f_\Phi(t) = \prod_{\tau\in \Phi} (t - \tau(\pi)) \in \mc{O}_{E_\Phi, p}[t].\]
        The image of $\mc{O}_{E_\Phi} \times_\Z \mc{O}_{E, p}$ in $\widetilde{E_{\Phi, p}}$ is $\mc{O}_{E_\Phi, p}[t]/f_\Phi(t)$. This means $d_{\Phi, p}$ is the discriminant $f_\Phi(t)$, or
        \[d_{\Phi, p} = \prod_{(\tau, \tau')} (\tau(\pi) - \tau'(\pi))^2,\]
        where the product is taken over all unordered pairs of distinct $\tau \neq \tau' \in \Phi$. We can write the summation over all $\Phi \supset \phib$ as the sum of $\log \abs{\tau(\pi) - \tau'(\pi)}_\sigma$ over all pairs $(\tau, \tau')$, and then subtract the pairs when $\tau = \conj{\tau'}$ and when $\tau \in \phib$ and $\tau' \in \conj{\phib}$, or vice versa. Thus, we can simplify the sum as
        \begin{align*}
        \sum_{\Phi \supset \phib} \log \abs{d_{\Phi, p}d_{\conj{\Phi}, p}}_\sigma =&\log \abs*{\frac{\dprod_{(\tau, \tau') \in \Hom(E, \clos{\Q_p})} (\tau(\pi) - \tau'(\pi))^2}{\dprod_{\tau \in \Phi} (\tau(\pi) - \conj{\tau}(\pi))^2}}_\sigma^{2^{\abs{\comp{\Sigma}} - 1}}\\
        &+ \log \abs*{\frac{\dprod_{\tau \in \phib} (\tau(\pi) - \conj{\tau}(\pi))^2}{\dprod_{(\tau, \tau') \in \phib \sqcup \conj{\phib}} (\tau(\pi) - \tau'(\pi))^2}}_\sigma^{2^{\abs{\comp{\Sigma}} - 1}}\\
        &+ \log\abs*{\dprod_{(\tau, \tau') \in \phib} (\tau(\pi) - \tau'(\pi))^2(\conj{\tau}(\pi) - \conj{\tau'}(\pi))^2}_\sigma^{2^{\abs{\comp{\Sigma}}}}.
        \end{align*}
        We can simplify the first term as $\log \abs*{\frac{d_E, p}{d_{E/F, p}}}_\sigma = \log \abs{d_{F, p}}_\sigma^2$. The second term can be written as $\log \abs*{\frac{N_{F/E_X} d_{E/F}}{d_\Sigma}}_\sigma = \log \abs*{\frac{d_{E/F, \Sigma}}{\Sigma}}_\sigma$. Finally, the last term is $\log |d_{\phib, p}d_{\conj{\phib}, p}|_\sigma$.

        Plugging this back in gives
        \begin{align*}
            \frac{1}{2}h_{\widehat{\mc{L}_U}}(P_U) =& \frac{1}{2^{\abs{\comp{\Sigma}}}} \sum_{\Phi \supset \phib} h(\Phi) - \frac{\abs{\comp{\Sigma}}}{g2^g} \sum_\Phi h(\Phi)\\
            &+ \frac{1}{8} \log d_{E/F, \Sigma}d_\Sigma^{-1} + \frac{1}{4}\log d_\phib d_{\conj{\phib}} + \frac{1}{4g} \log d_B + \frac{\abs{\Sigma}}{4g} \log d_F.
        \end{align*}
    \end{proof}
	
	\section{Andr\'e--Oort for Shimura Varieties}\label{sec:Andre Oort}
	
	A definition of the height of a partial CM-type was given by \cite{Pil21}. We show that their definition of the partial CM-type is compatible with our quaternionic height. We first recall their definition of the modified height on special points, specialized to the case of a partial CM-type. Let $E/F$ be a CM extension, with $[F:\Q] = g$, and set $R_E \coloneqq \Res_{F/\Q} E^\times/F^\times$. Let $\phib \subset \Hom(E, \C)$ be a partial CM-type and $\phib'$ be a complementary partial CM-type. Use them to identify $E\otimes_\Q \R \overset{\phib \sqcup \phib'}{\cong} \C^g$. Then we have that
	\[R_{E, \R} \cong \prod_{\sigma \in \phib \sqcup \phib'} \C_\sigma/\R,\]
	and we take our homomorphism $h_\phib \colon \C^\times \to R_{E, \R}$ as
	\[h_\phib(z) \coloneqq \paren*{\prod_{\sigma \in \phib} z_\sigma, \prod_{\sigma \in \phib'} 1_\sigma}.\]
	The Shimura datum $(R_E, h_\phib)$ and compact open subgroup $K \subset R_E(\mb{A}_f)$ give rise to a $0$-dimensional Shimura variety $T_K$ whose complex points are
	\[T_K(\C) \cong R_E(\Q) \backslash R_E(\mb{A}_f)/K.\]
	It has a canonical model over a number field $E_T$. We identify
	\[R_E(\C) \cong \prod_{\sigma \in \phib \sqcup \phib'} \C_\sigma.\]
	Let $\chi\colon R_E(\C) \to \C$ be the character given by
	\[\chi\paren*{\prod_{\sigma \in \Phi}z_\sigma} \coloneqq \prod_{\sigma \in \phib} \frac{z}{\conj{z}}.\]
	Let $V$ be the smallest $\Q$-representation of $R_E$ whose complexification contains $\chi$. Let $\Fil^aV$ be the smallest piece of the Hodge filtration and assume that it is one-dimensional and $R_E$ acts on it via $\chi$. Let $\Lambda \subset V$ be a maximal lattice and now take $K = \prod_p K_p \subset R_E(\mb{A}_f)$ to be the stabilizer of $\Lambda$. Let $\psi$ be a polarization on $V$ that takes integral values on $\Lambda$.
	
	The representation $V$ of $R_E$ gives rise to a vector bundle $\mc{V}_K$ over $T_K$ and filtration on it. By \cite{Dia22}, over every non-archimedean place $v$ of $\mc{O}_{E_T}$ lying over a prime $p \in \Z$, our data $(\Lambda, V, \Fil)$ can be functorially identified with data $(_p\Lambda, _pV, _p\Fil)$ of a filtered vector bundle over $T_{E_{T, v}, K}$, the Shimura variety extended over local fields.
	
	To each place $v$ of $\mc{O}_{E_T}$, we define a norm on $\Fil^a \Lambda$ as:
	\begin{itemize}
	    \item If $v \mid \infty$, then the norm is the Hodge norm given by the polarization $q$;
	    
	    \item If $v \mid p$ is a non-archimedean place such that
	    \begin{itemize}
	        \item $T$ is unramified at p,
	        \item $K_v$ is maximal, and
	        \item $p \ge a \dim V + 2$,
	    \end{itemize}
	    then use the crystalline norm on ${}_pV$;
	    
	    \item For all other places $v$, use the intrinsic norm on ${}_p\Lambda$.
	\end{itemize}
	Now the height of $\phib$ can be defined as
	\[h(\phib) \coloneqq \sum_{v} -log \|s\|_v,\]
	where $s$ is any element of $\mc{V}_K$ and the sum is over all places of $\mc{O}_{E_T}$.
	
	The height depends on the choice of lattice $\Lambda$ and polarization $q$, but only up to $d_E$, the discriminant of $E$.
	
	\begin{thm}[{{\cite[Lem. 9.4, Thm. 9.5, 9.6]{Pil21}}}]
	    The height $h(\phib)$ is defined up to $O(\log d_E)$.
	\end{thm}
	
	\begin{thm}
	    \[2h(\phib) = h_{\widehat{\mc{L}_U}}(P_U) + O(\log d_E).\]
	\end{thm}
	\begin{proof}
	    Consider the representation of $E$ on $V = B$ through left multiplication. Our point $P_U$ corresponds to an action whose trace is $\Tr_{\phib \sqcup \phib'} + \Tr_{\conj{\phib} \sqcup \phib'}$. When we take the Shimura variety associated with the adjoint group $G^{\ad}$, then this representation gives a representation of $R_E$ on $V/F$ whose trace is given by the trace of $\frac{z_\tau}{z_{\conj{\tau}}}$, meaning that we get $2\chi$. Thus, we get a representation of $R_E$ whose complexification contains $2\chi$. Thus, we are reduced to showing that the choice of lattices at each finite place are the same. However, since our equality is only up to $O(d_E)$, it suffices to consider primes where $B, E$ are unramified and the local norm used in the definition of $h(\phib)$ is given by the crystalline norm.

        Let $S = \mc{O}_{E_X'^p, p}$ again be the maximal unramified extension of $\mc{O}_{E_X', p}$. To show that the lattices coincide, it suffices to check the two lattices at each $S$ point of $X_U$. Under the mapping $X_U \to X_U \times Y_J \to X''_{U''}$, the point $P_U$ corresponds to an abelian variety $\mc{A}$ with complex multiplication of type $\phib \sqcup \phib' + \conj{\phib} \sqcup \phib'$. By \cite[Sec. 9.3]{Pil21}, the lattice given by the crystalline norm is the same as the lattice from integral de Rham cohomology. So the lattice at that point is
        \[\Omega(\mc{A}) \otimes S \cong \Omega(\mc{A}[p^\infty]) \otimes S \cong \Omega(\mc{H}''_S).\]
        However, by Proposition \ref{prop:compare p divisible groups}, this is the same as $\Omega(\mc{H}_S)$. Moreover, the pairing is perfect here meaning that we get the same lattice on $\Omega(\mc{A}^t) \otimes S \cong \Omega(\mc{H}''^t_S)$. Thus, twice $h(\phib)$ corresponds to taking the height relative to the lattice $\Omega(\mc{A}) \otimes \Omega(\mc{A}^t)$ which is $\Omega(\mc{H}_S) \otimes \Omega(\mc{H}_S^t)$ which by Theorem \ref{thm:Hodge Bundle p Divisible Group} is just $\mc{L}_U$, as required.
	\end{proof}
	
    \bibliographystyle{alpha}
    \bibliography{references}{}
\end{document}